\documentclass[smallextended,envcountsect,envcountsame]{svjour3}
\usepackage{graphicx}
\usepackage{fix-cm}
\usepackage{amsmath}
\usepackage{amssymb}
\usepackage{latexsym}
\usepackage{dsfont}
\usepackage{xcolor}

\def\disp{\displaystyle}

\def\Limsup{\mathop{{\rm Lim}\,{\rm sup}}}

\def\tto{\;{\lower 1pt \hbox{$\rightarrow$}}\kern -10pt
\hbox{\raise 2pt \hbox{$\rightarrow$}}\;}

\def\Hat{\widehat}
\def\Tilde{\widetilde}
\def\Bar{\overline}
\def\ra{\rangle}
\def\la{\langle}
\def\ve{\varepsilon}
\def\B{\mathbb{B}}

\def\R{\mathbb{R}}
\def\N{\mathbb{N}}

\def\ox{\bar{x}}

\def\co{\mbox{\rm co}\,}

\def\epi{\mbox{\rm epi}\,}

\def\dom{\mbox{\rm dom}\,}

\def\cl{\mbox{\rm cl}\,}

\def\cone{\mbox{\rm cone}\,}

\def\inter{\mbox{\rm int}\,}

\def\dn{\downarrow}
\def\O{\Omega}

\def\ph{\varphi}
\def\emp{\emptyset}
\def\st{\stackrel}
\def\oR{\Bar{\R}}
\def\N{\mathbb{N}}
\def\lm{\lambda}

\def\tt{\tilde t}
\def\dd{\delta}
\def\al{\alpha}

\def\Th{\Theta}

\newcounter{count}

\usepackage{cite}
\usepackage{dsfont}
\usepackage{enumitem}
\DeclareMathOperator{\subm}{\partial}
\DeclareMathOperator{\subf}{\hat{\partial}}

\DeclareMathOperator{\ri}{ri}

\newcommand{\Intf}[1]{{E}_{#1}}

\DeclareMathOperator{\sub}{\partial}
\DeclareMathOperator{\1}{\mathds{1}}
\newcommand{\T}{\Omega}
\newcommand{\Rex}{\overline{\mathbb{R}}}
\let\epsilon\varepsilon
\newcommand{\REVPP}[1]{{{#1}}}

\begin{document}

\title{New extremal principles with applications\\to stochastic and semi-infinite programming\thanks{Research of the first author was partly supported by the USA National Science Foundation under grants DMS-1512846 and DMS-1808978, by the USA Air Force Office of Scientific Research under grant \#15RT04, and by the Australian Research Council under Discovery Project DP-190100555. The second author was   partially supported by:  CONICYT  Grants: Fondecyt Regular 1190110 and Fondecyt Regular 1200283 and  Programa Regional Mathamsud 20-Math-08 CODE: MATH190003.}}

\author{Boris S. Mordukhovich \and \mbox{Pedro P\'erez-Aros}}

\institute{Boris S. Mordukhovich \at Department of Mathematics, Wayne State University, Detroit, Michigan 48202, USA\\ \email{boris@math.wayne.edu} \\
\and Pedro P\'erez-Aros \at Instituto de Ciencias de la Ingenier\'ia, Universidad de O'Higgins, Rancagua, Chile\\
\email{pedro.perez@uoh.com}}

\dedication{Dedicated to Marco L\'opez in honor of his 70th birthday, with great respect}

\date{Received: date / Accepted: date}

\maketitle

\begin{abstract}

This paper develops new extremal principles of variational analysis that are motivated by applications to constrained problems of stochastic programming and semi-infinite programming without smoothness and/or convexity assumptions. These extremal principles concern measurable set-valued mappings/multifunctions with values in finite-dimensional spaces and are established in both approximate and exact forms. The obtained principles are instrumental to derive via variational approaches integral representations and upper estimates of regular and limiting normals cones to essential intersections of sets defined by measurable multifunctions, which are in turn crucial for novel applications to stochastic and semi-infinite programming.\vspace*{-0.05in}

\vspace*{-0.05in}

\vspace*{-0.05in}
\end{abstract}

\section{Introduction}\label{intro}

Variational and extremal principles of modern variational analysis have been widely recognized as fundamental ingredients to deal with theoretical and numerical issues arising in optimization theory and its applications; see, e.g., the books \cite{m06,m18,rw} and the references therein. Despite numerous successful applications of variational principles and techniques to various classes of constrained optimization problems, some important areas are still largely underinvestigated, while advanced methods of variational analysis seem to be very appropriate and promising for required applications. Among such areas we mention broad classes of constrained problems in stochastic and semi-infinite programming. We refer the reader to \cite{gl,sdr} for fundamental aspects of these disciplines and to \cite{clmp1,clmp2,hhp,m18,mn,mn1,mn2,ppa1} for some publications that apply variational analysis and generalized differentiation to problems of such types.\vspace*{0.05in}

In this paper we study optimization problems given in the form
\begin{align}\label{firstproblem}
\begin{array}{c}
\mbox{minimize }\;h(x)\;\mbox{ subject to}\\
x\in M(\omega)\;\text{ for almost all }\;\omega\in\Omega,
\end{array}
\end{align}
where $(\Omega,\mathcal{A},\mu)$ is a $\sigma$-finite measure space, where $M\colon\Omega\tto\R^n$ is a measurable multifunction with closed values, and where $h\colon\R^n\to\Rex:=(-\infty,\infty]$ is a lower semicontinuous (l.s.c.) extended-real-valued function. The framework of \eqref{firstproblem} is quite general and includes---among other classes---robust optimization problems, bilevel programs, and semi-infinite programs with some uncertainties in the data of $M(\omega)$. It is obvious that problem \eqref{firstproblem} can be equivalently written in the unconstrained format:
\begin{equation}\label{formulation2}
\mbox{minimize }\;h(x)+\delta_{M_\cap}(x)\;\mbox{ over }\;x\in\R^n,
\end{equation}
where $M_\cap$ is the {\em essential intersection} of $M$ defined by
\begin{equation}\label{ess-inter}
M_{\cap}:=\big\{x\in\R^n\big|\;x\in M(\omega)\;\text{ for almost all }\;\omega\in\Omega\big\},
\end{equation}
and where $\dd_\Th(x)$ stands for the indicator function of the set $\Th$ that is equal to $0$ if $x\in\Th$ and $\infty$ otherwise. Note that the constrained problem \eqref{formulation2} is intrinsically nonsmooth, even when $h$ is a smooth function. As a rule of thumb, necessary optimality conditions for local minimizers of \eqref{formulation2} are formulated as
$$
0\in\partial h(x)+N(x;M_\cap)
$$
via appropriate subdifferential and normal cone notions under suitable qualification conditions. To proceed efficiently in this direction, we have to select adequate subdifferential and normal cone constructions and to be able to calculate (or at least to estimate from above) the normal cone to sets of type \eqref{ess-inter}. To the best of our knowledge, it has not been done in the literature, except for the cases where $\O$ consists of finitely many or countably many points.

The main goals of this paper are to establish {\em efficient calculus rules} of regular and limiting normal cones (see Section~\ref{notation} for the definitions) to the set $M_\cap$ from \eqref{ess-inter} generated by measurable multifunctions and then to apply the obtained results to deriving necessary optimal conditions in general constrained problems of {\em stochastic} and {\em semi-infinite programming}. These issues happen to be very challenging, and we accomplish our goals by establishing new {\em extremal principles} for measurable multifunctions that are certainly of their independent interest, besides the applications presented below. Our developments in this direction follow the lines of \cite{mp1} (see also \cite{m18}), where the notion of extremality and appropriate versions of the extremal principle were given for countable systems of sets. In the case of finitely many sets, these notions and results reduce to those originated in \cite{km} and then have been extensively developed and applied in variational analysis and optimization; see, e.g., \cite{m06,m18} with comprehensive commentaries and references therein. Note the sequential extremal principle obtained below for measurable multifunctions in new even for systems of countably many sets. The latter corresponds to the setting of \eqref{ess-inter}, where the set $\Omega$ consists of countably many points with the measure $\mu$ being atomic at these points. The case of an arbitrary measure $\mu$ and a $\mu$-measurable multifunction $M$ in \eqref{ess-inter} defined on an arbitrary set $\Omega$ allows us to cover in the framework of \eqref{firstproblem} general problems of stochastic programming, which has never been done before, and also to significantly extend the applications of \cite{mp2} from countable to general constraint systems in nonsmooth and nonconvex semi-infinite programming.\vspace*{-0.05in}

The rest of the paper is organized as follows. In Section~\ref{notation} we present some {\em constructions} and {\em preliminaries} from variational analysis and generalized differentiation that are widely used below.\vspace*{-0.05in}

Section~\ref{EXTREMALFORSETVALUED} is devoted to the introduction and study of new concepts of {\em extremality} for measurable multifunctions and deriving {\em extremal principles} for them. We establish two extremal principles that play crucial roles in deriving the subsequent calculus rules and applications. The first extremal principle addresses general measurable multifunctions with closed values and is expressed in the {\em sequential/approximating form} via regular normals at nearby random points. The second principle concerns measurable {\em cone-valued} multifunctions {\em extremal at the origin} and is given in the {\em exact form}, i.e., it is expressed in terms of the limiting normal cone exactly at the origin in $L^p(\Omega,\R^n)$, $1\le p<\infty$, as the extremal point. The statements of both extremal principles involve integrals over $\Omega$ with respect to the given measure on $\Omega$.\vspace*{-0.07in}

In Section~\ref{SectionIntegralRepresentation} we develop a {\em variational approach}, based on employing the obtained extremal principles and related variational results, to derive integral representations and upper estimates of regular and limiting normals to essential intersections of measurable multifunctions with the main results obtained here for {\em cone-valued} measurable mappings.\vspace*{-0.05in}

The next Section~\ref{chip} extends this approach to evaluating the normal cones to essential intersections \eqref{ess-inter} of {\em arbitrary measurable multifunctions} with closed values in finite-dimensional spaces by involving in addition an appropriate extension of the so-called {\em conical hull intersection property} (CHIP) to the case of measurable multifunctions that is introduced and studied in this section. A typical {\em calculus rule} of this type is given by
\begin{align*}
N(\bar{x};M_\cap)\subset{\rm cl}\Big(\int_{\Omega} N\big(\bar{x};M(\omega)\big)d\mu(\omega)\Big)
\end{align*}
in terms of the closure of the {\em Aumann integral} of set-valued mappings. The obtained calculus results are crucial for the subsequent applications.

Section~\ref{ApplicationSTOCHASTIC} is devoted to applications of the results developed above to general problems of {\em stochastic programming}. First we derive necessary optimality conditions for nonsmooth and nonconvex stochastic programs with random constraints described by measurable set-valued mappings  $M\colon\Omega\tto\R^n$. Then we specify these conditions in the case of stochastic programs with inequality constraints under appropriate constraint qualifications. All the obtained qualification and optimality conditions are expressed in terms of limiting normals and subgradients calculated precisely at the local minimizer in question.

Section~\ref{AplSIP} concerns general problems of {\em semi-infinite programming} with nonsmooth and nonconvex data and index sets given by an arbitrary metric space. Similarly to Section~\ref{ApplicationSTOCHASTIC}, we derive pointwise necessary optimality conditions for such problems considering first programs with set-valued constraints and then specifying the results in the case of infinite inequality systems.\vspace*{-0.2in}

\section{Preliminaries from variational analysis}\label{notation}\vspace*{-0.1in}

In this section we present some preliminaries from variational analysis and generalized differentiation that are broadly used in what follows.
Our notation and terminology are standard; see, e.g., \cite{m18,rw}. Recall that $\mathbb{B}$ stands for the closed unit ball of the finite-dimensional Euclidean space in question, that $\mathbb{B}_r(x):=x+r\mathbb{B}$ for $x\in\R^n$ and $r>0$, and that $\N:=\{1,2,\ldots\}$. Given a nonempty set $\Th\subset\R^n$, we use the symbols $\inter\Th$, $\ri\Th$, $\cl\Th$, $\co\Th$, and $\cone\Th$ to denote the {\em interior, relative interior, closure, convex hull}, and {\em conic hull} of $\Th$, respectively. The symbol $^*$ indicates the {\em duality} correspondence. In particular, $\Th^*:=\{v\in\R^n|\;\la v,x\ra\le 0\;\mbox{ for all }\;x\in\Th\}$, and $A^*$ stands for the matrix transposition (adjoint operator). The {\em distance function} of $\Th$ is denoted by $d_\Th(x):=\inf_{u\in\Th}\|x-u\|$ for all $x\in\R^n$.

Given further a set-valued mapping $F\colon\R^n\tto\R^m$, define the (Painlev\'e-Kuratowski) {\em outer limit} of $F$ as $x\to\ox$ by
\begin{equation*}
\Limsup_{x\to\ox}F(x):=\big\{y\in\R^m\big|\;\exists\,x_k\to\ox,\;y_k\to y\;\mbox{ with }\;y_k\in F(x_k),\;k\in\N\big\}.
\end{equation*}

In this paper we use the following collections of generalized normals to arbitrary sets. The (Fr\'echet) {\em regular normal cone} to $\Th$ at $\ox\in\Th$ is defined by
\begin{equation}\label{rn}
\Hat N(\ox;\Th):=\Big\{x^*\in\R^n\Big|\;\limsup_{x\st{\Th}{\to}\ox}\frac{\la x^*,x-\ox\ra}{\|x-\ox\|}\le 0\Big\},
\end{equation}
where the symbol $x\st{\Th}{\to}\ox$ means that $x\to\ox$ with $x\in\Th$. The (Mordukhovich) {\em basic/limiting normal cone} to $\Th$ at $\ox\in\Th$ is defined by
\begin{equation}\label{ln}
N(\ox;\Th):=\Limsup_{x\st{\Th}{\to}\ox}\Hat N(x;\Th).
\end{equation}
Recall the well-known duality relation $\Hat N(\ox;\Th)=T^*(\ox;\Th)$ between \eqref{rn} and the (Bouligand-Severi) {\em tangent/contingent cone} to $\Th$ at $\ox$ given by
\begin{equation*}
T(\bar{x};\Th):=\Limsup\limits_{\tau\downarrow 0}\frac{\Th-\bar{x}}{\tau}.
\end{equation*}
Note that, due to its nonconvexity, the limiting normal cone \eqref{ln} cannot be dual to any tangential approximation of $\Th$ at $\ox$. Nevertheless, the normal cone \eqref{ln} and the associated subdifferential and coderivative constructions for functions and mappings enjoy comprehensive calculus rules based on variational/extremal principles of variational analysis; see \cite{m06,m18,rw}. The set $\Th$ is called {\em normally regular} at $\ox\in\Th$ if $\Hat N(\ox;\Th)=N(\ox;\Th)$.

Let $f\colon\R^n\to\oR$ be an extended-real-valued function with the {\em domain} $\dom f:=\{x\in\R^n|\;f(x)<\infty\}$ and the {\em epigraph} $\epi f:=\{(x,\al)\in\R^{n+1}|\;\al\ge f(x)\}$. The (Fr\'echet) {\em regular subdifferential} of $f$ at $\ox\in\dom f$ is given by
\begin{equation}\label{fs}
\Hat\partial f(\ox):=\big\{x^*\in\R^n\big|\;(x^*,-1)\in\Hat N\big((\ox,f(\ox));\epi f\big)\big\}.
\end{equation}
Using now the limiting normal cone \eqref{ln}, we define the limiting subdifferential constructions known as the (Mordukhovich) {\em basic subdifferential} and {\em singular subdifferential} of $f$ at $\ox\in\dom f$, respectively:
\begin{align}
\partial f(\ox):=\big\{x^*\in\R^n\big|\;(x^*,-1)\in N\big((\ox,f(\ox));\epi f\big)\big\},\label{bs}\\
\partial^\infty f(\ox):=\big\{x^*\in\R^n\big|\;(x^*,0)\in N\big((\ox,f(\ox));\epi f\big)\big\}.\label{ss}
\end{align}
The construction $\Hat\partial^\infty f(\ox)$ is defined similarly to \eqref{ss} by using $\Hat N$ therein.\vspace*{0.02in}

If $f$ is convex, then \eqref{fs} and \eqref{bs} reduce to the subdifferential of convex analysis. If $f$ is l.s.c.\ around $\ox$, then the condition $\partial^\infty f(\ox)=\{0\}$ fully characterizes the local {\em Lipschitz continuity} of $f$ around this point. We refer the reader to the books \cite{m06,m18,rw} and the bibliographies therein for various results and applications of the subdifferential constructions \eqref{fs}--\eqref{ss} including {\em full calculi} for the limiting ones \eqref{bs} and \eqref{ss}.\vspace*{0.02in}

Next we consider a complete $\sigma$-finite measure space $({\Omega},\mathcal{A},\mu)$ with $\mu(\Omega )>0$. For any $p\in[1,\infty]$, denote by $\|\cdot\|_p$ the norm of the classical Lebesgue space $L^p(\Omega,\mathbb{R}^n)$. A set-valued mapping $M\colon{\Omega}\tto\R^n$ is said to be {\em measurable} if for every open set $U\subset\mathbb{R}^n$ the inverse image
$M^{-1}(U)$ is measurable, i.e., $M^{-1}(U)\in\mathcal{A}$. The {\em essential intersection} $M_{\cap}$ of $M$ was defined in \eqref{ess-inter}. Recall also that the (Aumann) {\em integral} of $M:{\T}\tto\R^n$ over $A\in\mathcal{A}$ is given by
\begin{equation*}
\int_A M(\omega)d\mu(\omega):=\left\{\int_A x^*(\omega)d\mu(\omega)\Bigg|x^*\in{L}^1({\Omega},\R^n)\textnormal{ and }x^*(\omega)\in M(\omega)\text{  a.e.}\right\}.
\end{equation*}

Let us now formulate two known results on subdifferentiation of integral functionals needed in  what follows. The first result is classical in convex analysis of integral functionals; see, e.g., \cite[Chapter~14]{rw}. A mapping $f\colon\Omega\times\mathbb{R}^n\to\Rex$ is called a {\em normal integrand} if it is $\mathcal{A}\otimes\mathcal{B}(\mathbb{R}^n)$-measurable (where $\mathcal{B}(\mathbb{R}^n)$ is the \emph{Borel $\sigma$-algebra}, i.e., the $\sigma$-algebra generated by all open sets of $\mathbb{R}^n$), and if $f_\omega:=f(\omega,\cdot)$ is l.s.c.\ for every $\omega\in{\Omega}$. If in addition $f_\omega$ is convex for all $\omega\in\Omega$, then it is said to be a {\em convex normal integrand}.\vspace{-0.05in}

\begin{proposition}[generalized Leibniz rule for convex integrals]\label{propositiondifconvinte} Given a convex normal integrand $f\colon\Omega\times\mathbb{R}^n\to\oR$, define the integral $\Intf{f}(x):=\int_\Omega f(\omega,x)d\mu(\omega)$. If $\ox$ is a point where $\Intf{f}$ is continuous, then we have
\begin{align}\label{eq:subIntegral}
\sub\Intf{f}(\ox)=\int_\Omega\sub f_\omega(\ox)d\mu(\omega).
\end{align}
Hence $\Intf{f}$ is differentiable at $\ox$ if the right-hand side of \eqref{eq:subIntegral} is a singleton.
\end{proposition}\vspace*{-0.05in}

The second result has been recently established in \cite{chp}; it provides a sequential evaluation of regular subgradients of integral functionals involving nonconvex normal integrands.\vspace*{-0.05in}

\begin{proposition}[sequential subdifferentiation of nonconvex integral functionals]\label{theoremsubdiferential} Let $\mu$ be a finite measure on $\O$, and let $f\colon\Omega\times\mathbb{R}^n\to[0,\infty]$ be a normal integrand. Take $p,q\in[1,\infty]$ with $1/p+1/q=1$. Then for every  $x^*\in\Hat{\sub}\Intf{f}(\ox)$ with $\ox\in\dom\Intf{f}$ there exist sequences of elements $y_k\in\mathbb{R}^n$, $x_k\in{L}^p(\Omega,\mathbb{R}^n )$, and $x_k^*\in{L}^q(\Omega,\mathbb{R}^n)$ as $k\to\infty$ such that:\\[1ex]
{\bf(i)} $x_k^*(\omega)\in\Hat{\sub}f(\omega,x_k(\omega))$ a.e., $\|\ox-y_k\|\to 0$, $\|\ox-x_k(\cdot)\|_p\to 0$;\\[1ex]
{\bf(ii)} $\displaystyle\int_{\Omega}\|x_k^*(\omega)\|\cdot\|x_k(\omega)-y_k\|d\mu(\omega)\to 0$, $\displaystyle\int_{\O}\langle x_k^*(\omega),x_k(\omega)-\ox\rangle d\mu(\omega)\to 0$;\\[1ex]
{\bf(iii)}$\displaystyle\int_{\O}x_k^*(\omega)d\mu(\omega)\to x^*$, $\displaystyle\int_{\O}|f(\omega,x_k(\omega))-f(\omega,\ox)|d\mu(\omega)\to 0$.
\end{proposition}\vspace*{-0.05in}

The final result of this section provides simple subdifferential relations concerning the \REVPP{distance functions to cones}.\vspace*{-0.1in}

\begin{proposition}[subdifferentiation of distance functions for cones]\label{inclusionpostivehomogeneous} Let $K\subset\R^n$ be a closed cone. Then we have the inclusions
\begin{equation*}
\Hat\partial d_K(\ox)\subset\partial d_K(0)\;\mbox{ for all }\;\ox\in\R^n\;\mbox{ and }\;\Hat N(\ox;K)\subset N(0;K)\;\mbox{ for all }\;\ox\in K.
\end{equation*}
\end{proposition}
\begin{proof} $\:$
Picking $x^*\in\Hat\partial d_K(\ox)$ gives us
\begin{equation*}
\liminf_{x\to\ox}\frac{d_K(x)-d_K(\ox)-\langle x^*,x-\ox\rangle}{\|x-\ox\|}\ge 0.
\end{equation*}
Since the mapping $x\mapsto d_K(x)$ is positive homogeneous, for every $s>0$ we get
\begin{equation*}
\frac{d_K(sx)-d_K(s\ox)-\langle x^*,sx-s\ox\rangle}{\|sx-s\ox\|}= \frac{d_K(x)-d_K(\ox)-\langle x^*,x-\ox\rangle}{\|x-\ox\|},
\end{equation*}
which implies by denoting $\tilde{x}:=sx$ the following inequality:
\begin{equation*}
\liminf_{\tilde{x}\to s\ox}\frac{d_K(\tilde{x})-d_K(s\ox)-\langle x^*,\tilde{x}-s\ox\rangle}{\|\tilde{x}-s\ox\|}\ge 0
\end{equation*}
that ensures in turn that $x^*\in\subf d_K(s\ox)$. By passing to the limit $s\to 0$, we readily arrive at $x^*\in\subm d_K(0)$.
	
The second claimed inclusion follows from the relationships between the regular and limiting subdifferentials of the distance function and the corresponding normal cones; see, e.g., \cite[Corollary~1.96 and Theorem~1.97]{m06}.
\end{proof}\vspace*{-0.2in}

\section{Extremal principles for measurable set-valued mappings}\label{EXTREMALFORSETVALUED}\vspace*{-0.1in}

The concept of {\em extremality} for {\em finitely} many sets and the {\em extremal principle} for them were first formulated by Kruger and Mordukhovich \cite{km}; see also \cite{m94} where this notion was coined and \cite{m06,m18} for further developments, references, and applications. Such an extremal principle formulated via the limiting normal cone \eqref{ln} can be viewed as a far-going {\em variational} counterpart of the classical separation theorem in the case of {\em nonconvex} sets. Various extensions of this extremal principle to {\em countably} many sets can be found in \cite{kl,m18,mp1,mp2}.\vspace*{0.05in}

Following the line of \cite{mp1}, we introduce a new notion of extremality for measurable mappings and obtain an extremal principle for this notion.\vspace*{-0.07in}

\begin{definition}[local extremality for set-valued mappings]\label{local_extremal_at}
Consider a measure space $(\Omega,\mathcal{A},\mu)$ and a measurable set-valued mapping $M\colon\Omega\tto\mathbb{R}^n$, and let $M_\cap$ be taken from \eqref{ess-inter}. Then $M(\cdot)$ is said to be \emph{locally extremal} at $\ox\in M_\cap$ in $L^p(\Omega,\R^n)$ with some $p\in(1,\infty)$ if there exists a sequence of $a_k\in L^p(\Omega,\R^n)$ with $\|a_ k(\cdot)\|_p\to 0$ as $k\to\infty$ and an (open) neighborhood $U$ of $\ox$ such that for all $k\in\N$ we have
\begin{equation}\label{ext}
\bigcap_{\omega\in\Omega \text{ a.e. }}\big(M(\omega)-a_k(\omega)\big)\cap U=\emp,
\end{equation}
where the notation in the left-hand side of \eqref{ext} means that
\begin{equation*}
\bigcap_{\omega\in\Omega \text{ a.e.} }\big(M(\omega)-a_k(\omega)\big)\cap U:=\big\{x\in U\big|\;x\in M(\omega)-a_k(\omega)\text{ a.e.} \big\}.
\end{equation*}
\end{definition}

The crucial result of this paper establishes necessary conditions for extremality in the sense of Definition~\ref{local_extremal_at}. This novel extremal principle for measurable multifunctions is basic for the subsequent applications to generalized differential calculus of integral functionals derived via a variational approach, as well as to necessary conditions for general constrained problems of stochastic and semi-infinite programming. It is expressed in terms of sequences and involves regular normals to the values of the given set-valued mapping $M(\cdot)$. Note that the extremal principle of the following theorem is new even for the case of countably many sets considered in \cite{m18,mp1}, where this result was not established. In the case of finitely many sets, the obtained sequential extremal principle can be equivalently reduced to the exact one given in \cite{km,m06,m18}.\vspace*{-0.07in}

\begin{theorem}[sequential extremal principle]\label{fuzzy_extremal_prin} Let $M\colon\Omega\tto\mathbb{R}^n$ be a closed-valued measurable multifunction with respect to a finite measure $\mu$. Assume that $M$ is locally extremal at $\ox$ in $L^p(\Omega,\R^n)$ with some $p\in(1,\infty)$, and that the following {\sc nonoverlapping condition} holds at $\ox$: there exists neighborhood $U$ around $\ox$ such that
\begin{equation}\label{nonoverlapping}
\bigcap_{\omega\in\Omega \text{ a.e.}}M(\omega)\cap U=\{\ox\}.
\end{equation}
Then we get the {\sc sequential extremal principle} in $L^p(\Omega,\R^n)$ meaning that there exist sequences of $x_k^*\in L^q(\Omega,\R^n)$ and $x_k\in L^p(\Omega,\R^n)$ satisfying the conditions $x_k^*(\omega)\in\Hat N(x_k(\omega);M(\omega))$ a.e., $\|x_k(\cdot)-\ox\|_p\to 0$ as $k\to\infty$,
\begin{equation*}
\int_\Omega x_k^*(\omega)d\mu(\omega)=0,\;\text{ and }\;\|x_k^*\|_q=1\;\mbox{ for all }\;k\in\N,
\end{equation*}
where $\frac{1}{p}+\frac{1}{q}=1$. Furthermore, \REVPP{we can find $\epsilon_k \to0^+$ such that
\begin{equation}\label{est}
\|x_k(\omega)-\ox\|\le 2\|a_k(\omega)\| + \epsilon_k\;\text{ a.e. }\;\omega\in\Omega,\;k\in\N.
\end{equation}}
\end{theorem}
\begin{proof} $\;$ For each $k\in\N$ define the function
\begin{equation}\label{ph}
\varphi_k(x):=\int_{\Omega}d^p_{M(\omega)}\big(x+a_k(\omega)\big)d\mu(\omega)+\delta_{{\rm\small cl}\,U}(x),\quad x\in\R^n,
\end{equation}
where $a_k\in L^p(\Omega,\mathbb{R})$ and a neighborhood $U$ of $\ox$ taken from Definition~\ref{local_extremal_at}. We also assume that $U$ is the one for which the nonoverlapping condition \eqref{nonoverlapping} holds. Let us split the subsequent proof into six claims.\\[1ex]
{\bf Claim~1:} {\em For each $k\in\N$ the function $\varphi_k$ from \eqref{ph} is proper, l.s.c.,\ and attains its minimum on $\R^n$}.\\[1ex]
To verify the claim, observe first that due to the fact that $\ox\in M(\omega)$ a.e.\
\begin{equation}\label{ineq01}
\varphi_k(\ox)\le\int_{\Omega}\|a_k(\omega)\|^p d\mu(\omega)<\infty.
\end{equation}
Furthermore, for fixed $k\in\N$ and for any sequence $z_j\to x$ we get by using Fatou's lemma that
\begin{align*}
\varphi_k(x)&=\int_{\Omega}d_{M(\omega)}^p\big(x+a_k(\omega)\big)d\mu(\omega)+\delta_{{\rm cl}\,U}(x)\\&\le\liminf_{j\to\infty}\left(\int_{\Omega} d^p_{M(\omega)}\big(z_j+a_k(\omega)\big)d\mu(\omega)\right)+\liminf_{j\to\infty}\delta_{{\rm\small cl}\,U}( z_j)\\
&\le\liminf_{j\to\infty}\varphi_k(z_j),
\end{align*}
which shows that the function $\varphi_k$ is proper and l.s.c.\ Since $U$ is bounded, it follows that $\varphi_k$ attains its minimum on $\R^n$.\\[1ex]
{\bf Claim~2:} {\em Let $\hat{x}_k$ be a minimizer of $\varphi_k$. Then $\varphi_k(\hat{x}_k)>0$, $\hat{x}_k\to\ox$, and $\varphi_k(\hat{x}_k)\to 0$ as $k\to\infty$}.\\[1ex]
Indeed, due to the construction of $\ph_k$ in \eqref{ph} we have $\hat{x}_k\in\cl U$ for all $k\in\N$, which yields the boundedness of $\{\hat{x}_k\}$. Moreover, it follows from the extremality condition \eqref{ext} that $\varphi_k(\hat{x}_k)>0$ as $k\in\N$, since the negation of it tells us that $\hat{x}_k\in M(\omega)-a_k(\omega)$ for almost all $\omega\in\Omega$, a contradiction. Considering now a cluster point $\hat{x}$ of $\{\hat{x}_k\}$, we assume without relabeling that $\hat{x}_k\to\hat x$ and $a_k(\omega)\to 0 $ a.e.\ (recall that $\|a_k(\cdot)\|_p\to 0$) as $k\to\infty$; therefore
\begin{equation*}
d_{M(\omega)}^p(\hat{x})=\liminf_{k\to\infty}d^p_{M(\omega)}\big(\hat{x}_{k}+a_{k}(\omega)\big).
\end{equation*}
Hence, by employing Fatou's lemma again, we get
\begin{align*}
\int_{\Omega}d_{M(\omega)}^p(\hat{x})d\mu(\omega)\le&\int_{\Omega}\liminf_{k\to\infty}d^p_{M(\omega)}\big(\hat{x}_{k}+a_{k}(\omega)\big)d\mu(\omega)\\\le& \liminf_{k\to\infty}\int_{\Omega}d^p_{M(\omega)}\big(\hat{x}_{k}+a_{k}(\omega)\big)d\mu(\omega)\le\liminf_{k\to\infty}\varphi_{k}(\hat x_{k})\\  \le&\liminf_{k\to\infty}\varphi_k(\ox)\le\lim_{k\to\infty}\int_{\Omega}\|a_k(\omega)\|^p d\mu(\omega)=0,
\end{align*}
where in the last line we used \eqref{ineq01} and the fact that $\|a_k(\cdot)\|_p \to 0$. This implies that $\varphi_{k}(\hat{x}_k)\to 0$ and $\hat{x}\in M(\omega)$ for almost all $\omega\in\Omega$, which ensures by the nonoverlapping condition \eqref{nonoverlapping} that $\hat{x}=\ox$. From now on we suppose without loss of generality that $\hat{x}_k\in U$ for all $k\in\N$.\\[1ex]
{\bf Claim~3:} {\em There exists a sequence of measurable selections $x_k(\omega)\in M(\omega)$ such that for all $k\in\N$ we have
\begin{align}\label{Fuzzy:022}
\|\hat{x}_k+a_k(\omega)-x_k(\omega)\|=d_{M(\omega)}\big(\hat{x}_k+a_k(\omega)\big)\;\text{ for a.e. }\;\omega\in\Omega,
\end{align}
$x_k\in L^p(\Omega,\mathbb{R}^n)$, and $\|x_k(\cdot)-\ox\|_p\to 0$ as $k\to\infty$ with estimate \eqref{est}.}\\[1ex]
Indeed, it follows from, e.g., \cite[Theorem~14.37]{rw} that for each $k\in\N$ there exists a measurable selection $x_k(\omega)\in M(\omega)$ satisfying \eqref{Fuzzy:022}. Furthermore
\begin{align*}
\|\ox-x_k(\omega)\|&\le\|\hat{x}_k+a_k(\omega)- x_k(\omega)\|+\|\hat{x}_k-\ox\|+\|a_k(\omega)\|\\
&=d_{M(\omega)}\big(\hat{x}_k+a_k(\omega)\big)+\|\hat{x}_k-\ox\|+\|a_k(\omega)\|\\
&\le\|\hat{x}_k+a_k(\omega)-\ox\|+\|\hat{x}_k-\ox\|+\|a_k(\omega)\|\\
&\le 2\|\hat{x}_k-\ox\|+2\|a_k(\omega)\|\;\mbox{ for almost all }\;\omega\in\O.
\end{align*}
This readily yields estimate \eqref{est} considering $\epsilon_k:=2\|\hat{x}_k-\ox\|$, and also ensures that
\begin{align*}
\int_{\Omega}\|\ox-x_k(\omega)\|^p d\mu(\omega)&\le2^{2p-1}\left(\|\hat{x}_k-\ox\|^p\mu(\Omega)+\int_{\Omega}\|a_k(\omega)\|^p d\mu(\omega)\right),
\end{align*}
which verifies the claimed properties of the measurable selections $x_k(\omega)$.\\[1ex]
{\bf Claim~4:} {\em For each $k\in\N$ the function
\begin{align*}
\psi_k(x):=\int_{\Omega}\psi_{\omega,k}(x)d\mu(\omega)+\delta_{{\rm\small cl}U}(x)\;\mbox{ with }\;\psi_{\omega,k}(x):=\|x+a_k(\omega)-x_k(\omega)\|^p
\end{align*}
admits a minimizer $\hat{x}_k$ over the whole space $\R^n$}.\\[1ex]
To verify this claim, observe that
\begin{equation*}\psi_k(x)\ge\varphi_k(x)\ge\varphi_k(\hat{x}_k)=\psi_k(\hat{x}_k)\;\mbox{ for all }\;x\in\mathbb{R}^n,
\end{equation*}
which tells us that $\hat{x}_k$ is a minimizer of $\psi_k$ on $\R^n$.\\[1ex]
{\bf Claim~5:} {\em For every $k\in\N$ there exists a measurable selection $u^*_k(\omega)\in\sub\psi_{\omega,k}(\hat{x}_k)$ such that $u_k^*(\cdot)\in L^q(\Omega,\mathbb{R}^n)$,
\begin{align}\label{eq_claim5}
\int_{\Omega}u^*_k(\omega)d\mu(\omega)=0,\;\text{ and }\;\int_{\Omega}\|u^*_k\|^q(\omega)d\mu(\omega)>0.
\end{align}}
To verify it, recall from Claim~4 that $\hat{x}_k$ is a minimizer of the function $\psi_k$ defined therein. Employing then Proposition~\ref{propositiondifconvinte} and the subdifferential Fermat rule, with taking into account that $\hat x_k\in U$, gives us a measurable selection $u^\ast_k(\omega)\in\sub\psi_{\omega,k}(\hat{x}_k)$ such that $\int_{\Omega}u^*_k(\omega)d\mu(\omega)=0$. Define further the set
\begin{align*}
A_k:=\big\{\omega\in\Omega\big|\;d_{M(\omega)}(\hat{x}_k+a_k(\omega))>0\big\}
\end{align*}
 and deduce from Claim~2 that $\mu(A_k)>0$. Moreover, we have
\begin{align*}
u^*_k(\omega)=p\|\hat{x}_k+a_k(\omega)-x_k(\omega)\|^{p-1}\frac{\hat{x}_k+a_k(\omega)-x_k(\omega)}{d_{M(\omega)}\big(\hat{x}_k+a_k(\omega)\big)}\;\text{ for a.e. }\;\omega\in A_k.
\end{align*}
On the other hand, $u_k^*(\omega)=0$ for almost all $\omega\in\Omega\backslash A$, which yields
\begin{align*}
\int_{\Omega}\|u^*_k(\omega)\|^q d\mu(\omega)=p^q\varphi_k(\hat{x}_k).
\end{align*}	
Consequently, we get $u^*_k\in L^q(\Omega,\mathbb{R}^n)$, and hence \eqref{eq_claim5} holds.\\[1ex]
{\bf Claim~6:} {\em Define $x_k^*(\omega):=\frac{u_k^*(\omega)}{\|u_k^*\|_q}$, $k\in\N$. Then
\begin{align}\label{last_eq}
x^*_k(\omega)\in\Hat N\big(x_k(\omega);M(\omega)\big)\text{ a.e. on }\;\Omega\;\text{ and }\int_{\Omega}x^*(\omega)d\mu(\omega)=0.
\end{align}}
Indeed, it follows from \eqref{Fuzzy:022} that $\hat{x}_k+a_k(\omega)-x_k(\omega)\in\Hat N\big(x_k(\omega);M(\omega)\big)$ a.e.\ on $\Omega$; see, e.g., the statement and proof in \cite[Theorem~1.6, Step~1]{m18}. Since $\Hat N\big(x_k(\omega);M(\omega)\big)$ is a cone, we get $x^*_k(\omega)\in\Hat N\big(x_k(\omega);M(\omega)\big)$ a.e.\ on $\Omega$. Furthermore, \eqref{eq_claim5} tells us that the function $x^*_k$ is well-defined and satisfies the second part of \eqref{last_eq}. This completes the proof of the theorem.
\end{proof}\vspace*{0.05in}

Next we consider measurable multifunctions with {\em cone values} and define for them another notion of extremality, which extends the one from \cite{mp1} formulated for countable systems of cones.\vspace*{-0.05in}

\begin{definition}[\bf conic extremality at the origin]\label{definition_extra_origin}
Let $(\Omega,\mathcal{A},\mu)$ be a measure space, and let $\Lambda\colon\Omega\tto\mathbb{R}^n$ be a measurable multifunction with cone values. We say that $\Lambda(\cdot)$ is {\em extremal at the origin} in $L^p(\Omega,\R^n)$ with some $p\in(1,\infty)$ if there exists $a(\cdot)\in L^p(\Omega,\R^n)$ such that
\begin{equation*}
\bigcap_{\omega\in\Omega \text{ a.e. }}\big(\Lambda(\omega)-a(\omega)\big):=\big\{x\in\R^n\big|\;x\in\Lambda(\omega)-a(\omega)\text{ a.e.}\big\}=\emp.
\end{equation*}
\end{definition}\vspace*{-0.02in}

The following result provides an extension of \cite[Theorem 4.2]{mp1} from countable set systems to measurable multifunctions. In contrast to Theorem~\ref{fuzzy_extremal_prin}, we now obtain the result in terms of the limiting normal cone \eqref{ln} calculated {\em exactly} at the extremal point $\ox=0$; this motivates the name of the result.\vspace*{-0.05in}

\begin{theorem}[\bf exact extremal principle for cone-valued multifunctions]\label{EXTREMALPRINCIPLE}
Let $\Lambda\colon\Omega\tto\mathbb{R}^n$ be a measurable multifunction defined on a finite measure space and taking closed cone values. Assume that $\Lambda$ is extremal at $0\in L^p(\Omega,\R^n)$ with some $p\in(1,\infty)$, and that the nonoverlapping condition
\begin{equation*}
\bigcap_{\omega\in\Omega \text{ a.e.}}\;\Lambda(\omega)=\{0\}
\end{equation*}
fulfills. Then the $($exact$)$ {\sc conic extremal principle} holds in $L^p(\Omega,\R^n)$ with $\frac{1}{p} +\frac{1}{q} =1$, i.e., there exists $x^*(\cdot)\in L^q(\Omega,\R^n)$ such that $x^*(\omega)\in N(0;\Lambda(\omega))$ for almost all $w\in\O$ together with the equalities
\begin{equation}\label{03}
\int_\Omega x^*(\omega)d\mu(\omega)=0\;\text{ and }\;\int_\Omega\|x^*(\omega)\|^q(\omega)d\mu(\omega)=1.
\end{equation}
Furthermore, we can find $w(\cdot)\in L^p(\Omega,\mathbb{R}^n)$ for which
\begin{equation*}
x^*(\omega)\in\Hat N\big(w(\omega);\Lambda(\omega)\big)\text{ a.e. }\omega \in \Omega.
\end{equation*}
\end{theorem}
\begin{proof} $\;$ Let us show that the conic extremality of the mapping $\Lambda$ imposed in Theorem~\ref{EXTREMALPRINCIPLE} implies that $\Lambda$ is locally extremal at the origin in the sense of Definition~\ref{local_extremal_at}. Indeed, take $\alpha_k\dn 0$  as $k\to\infty$ and define $a_k(\omega):=\alpha_k a(\omega)$. Then for all $k\in\N$ we get the relationship
\begin{align*}
\bigcap_{\omega\in\Omega\text{ a.e.}}\big(\Lambda(\omega)-a_k(\omega)\big)=\emp,
\end{align*}
which verifies the claim. Applying now Theorem~\ref{fuzzy_extremal_prin} to $\Lambda$ with taking into account Proposition~\ref{inclusionpostivehomogeneous} gives us $x^*(\omega)\in\Hat N(x_k(\omega);\Lambda(\omega))\subset N(0;\Lambda(\omega))$ such that the conditions in \eqref{03} are satisfied. This completes the proof.
\end{proof}\vspace*{0.05in}

It is easy to see that the case of countably many cones in \cite[Theorem~4.2]{mp1} and \cite[Theorem~2.9]{m18} follows from Theorem~\ref{EXTREMALPRINCIPLE} with $p=2$ by considering the measure space $(\N,\mathcal{P}(\N),\mu)$, where $\mathcal{P}(\N)$ denotes the power set of $\N$, and where $\mu$ is the atomic measure given by $\mu(\{m\}):=(2^m)^{-1}$, $m\in\N$. Note that the proofs in \cite{m18,mp1} are significantly different from  the one given above.\vspace*{-0.1in}

\begin{remark} $(${\em on nonoverlapping condition}$)$\label{rem-nonover} The nonoverlapping condition was introduced in \cite{mp1} for developing extremal principles for countably many sets. It is  needed to bypass the intrinsic infinite dimensionality of essential intersections. Observe that this condition is not so restrictive because, as shown in the next section, we can construct while proving calculus rules a family of sets that automatically satisfies the nonoverlapping property.
\end{remark}\vspace*{-0.1in}

In what follows we are going to focus on applications of the sequential extremal principle for measurable multifunctions established in Theorem~\ref{fuzzy_extremal_prin} while planning to present various applications of the conic extremal principle from Theorem~\ref{EXTREMALPRINCIPLE} in our subsequent work; cf.\  some developments in \cite{m18,mp1,mp2} for the case of countably many sets.\vspace*{-0.1in}
\vspace*{-0.2in}

\section{Normals to essential intersections  via optimization}\label{SectionIntegralRepresentation}\vspace*{-0.1in}

The major goal of this section is to obtain efficient upper estimates and exact formulas for generalized normals to essential intersections \eqref{ess-inter} for measurable multifunctions by using a {\em variational approach}, which is mainly based on the extremal principle established above. Some of the results obtained here concern {\em cone-valued} mappings, and then they will be used in the next section in connection with the {\em conical hull intersection property} (CHIP). \vspace*{0.05in}

We begin with presenting such a result employed in what follows. It provides {\em sequential} optimality conditions for problems of type \eqref{formulation2}.\vspace*{-0.08in}

\begin{lemma}[sequential optimality conditions]\label{Fuzzyoptimalitycondition} Let $\ox$ locally minimize an l.s.c.\ function $h\colon\R^n\to\oR$ subject to $x\in M(\omega)$ a.e., where $M\colon\O\tto\R^n$ is a closed-valued measurable multifunction on a $\sigma$-finite measure space $(\Omega,\mathcal{A},\mu)$. Then for any $p,q\in[1,\infty]$ with $1/p+1/q=1$ there exist sequences $y_k,z_k,z_k^*\in\mathbb{R}^n$, $x_k(\cdot)\in L^p(\Omega,\mathbb{R}^n)$, and $x_k^*(\cdot)\in L^q({\Omega},\mathbb{R}^n)$ satisfying the conditions\vspace*{-0.05in}
\begin{equation*}
\begin{array}{ll}
&z_k^*\in\Hat\partial h(z_k),\;x_k^*(\omega)\in\Hat N\big(x_k(\omega);M(\omega)\big)\mbox{\text{ a.e.}},\\
&z_k\overset{h}{\to}\ox,\;\|y_k-\bar{x}\|\to 0,\;\|x_k(\cdot)-\bar{x}\|_p\to 0,\\
&\disp\int_{\Omega}\|x_k^*(\omega)\|\cdot\|x_k(\omega)-y_k\|d\mu(\omega)\to 0,\;z_k^*+\int_\Omega x_k^*(\omega)d\mu(\omega)\to 0,
\end{array}
\end{equation*}\vspace*{-0.1in}
where the symbol $z_k\overset{h}{\to}\ox$ means that $z_k\to\ox$ with $h(z_k)\to h(\ox)$ as $k\to\infty$.
\end{lemma}
\begin{proof} $\;$ Assume without loss of generality that the measure $\mu$ is finite. We can always suppose that $\Omega$ is a subset of a larger set; e.g., the collections of all its subsets. Then Cantor's theorem from measure theory (see, e.g., \cite[Theorem 161, p. 276]{Generaltopology}) tells us the cardinal number of the latter set is strictly larger than that of $\Omega$, and so there exists $\omega_0\notin\Omega$. Picking such a point $\omega_0$, define the measure space $(\Tilde{\Omega},\Tilde{\mathcal{A}},\Tilde{\mu})$ as follows: $\Tilde{\Omega}:=\Omega\cup\{\omega_0\}$ and $\Tilde{\mathcal{A}}$ is the $\sigma$-algebra generated by $\mathcal{A}\cup\{\{\omega_0\}\}$, which is nothing else than $\Tilde{\mathcal{A}}=\mathcal{A}\cup\{A\cup\{\omega_0\}|\;A\in\mathcal{A}\}$ with the measure
\begin{align*}
\Tilde{\mu}(A):=\left\{\begin{array}{lc}
\mu(A)&\text{ if }\;\omega_0\notin A,\\
\mu(A\backslash\{\omega_0\})+1&\text{ if }\;\omega_0\in A.
\end{array}\right.
\end{align*}
Define now the integrand $f\colon\Tilde{\Omega}\times\mathbb{R}^n\to\Rex$ by
\begin{align*}
f(\omega,x):=\left\{\begin{array}{cc}
h(x)&\text{ if }\;\omega=\omega_0,\\
\delta_{M(\omega)}(x)& \text{ if }\;\omega\ne\omega_0
\end{array}\right.
\end{align*}
and consider the function $\Intf{f}(x):=\int_{\Tilde{\Omega}}f(\omega,x)d\Tilde{\mu}(\omega)$ for which $\Intf{f}(x)=h(x)$ if $ x\in M(\omega)$ a.e.\ and $\Intf{f}(x)=\infty $ otherwise. It is easy to see that $\bar{x}$ is a local minimizer of $\Intf{f}$. Thus the subdifferential Fermat rule yields $0\in\Hat\partial\Intf{f}(\bar{x})$. Applying finally Proposition~\ref{theoremsubdiferential} completes the proof of the lemma.
\end{proof}\vspace*{0.05in}

Now we are ready to derive integral upper estimates of regular and limiting normals to the essential intersection $M_\cap$ by using the {\em closure} operation.\vspace*{-0.05in}

\begin{theorem}[upper estimates of normals via integral closures]\label{COROLLARY:FUZZYEXPRESSION} Let $M\colon\Omega\tto\R^n$ be a measurable multifunction with closed cone values. Then
\begin{equation}\label{COROLLARY:FUZZYEXPRESSION:EQUATION}
N(x;M_{\cap})\subset\cl\left(\int_\Omega N\big(0;M(\omega)\big)d\mu(\omega)\right)\;\mbox{ for all }\;x\in\mathbb{R}^n.
\end{equation}
\end{theorem}
\begin{proof} $\;$ To justify \eqref{COROLLARY:FUZZYEXPRESSION:EQUATION}, let us first verify the inclusion with the regular normal cone on the left-hand side. Indeed, take any $x^*\in\Hat N(x;M_{\cap})$ with $x\in\R^n$ and for every $\ve>0$ find by definition \eqref{rn} such  $\eta\in (0,\ve)$ that the function
\begin{equation*}
y\mapsto-\langle x^*,y-x\rangle+\epsilon\|y-x\|+\delta_{\mathbb{B}_\eta(x)}(y)+\delta_{M_{\cap}}(y)
\end{equation*}
attains its minimum at $x$.	Applying Lemma~\ref{Fuzzyoptimalitycondition} gives us sequences $v_k^*\in\mathbb{B}$ and $x_k^*(\omega)\in\Hat N(x_k(\omega);M(\omega))$ a.e. for which
$\|-x^*+\epsilon v_k^*+\int_\O x_k^*(\omega)d\mu(\omega)\|\to 0$ as $k\to\infty$. Since $M(\omega)$ are cones, we deduce from Proposition~\ref{inclusionpostivehomogeneous} that $x_k^*(\omega) \in N(0;M(\omega))$ a.e.\ It implies therefore that
\begin{align*}
x^*\in\int_\Omega N\big(0;M(\omega)\big)d\mu(\omega)+2\epsilon\mathbb{B}.
\end{align*}
 Taking into account that $\epsilon>0$ was chosen arbitrarily, we arrive at the claimed inclusion \eqref{COROLLARY:FUZZYEXPRESSION:EQUATION}. The regular normal cone therein can be clearly replaced by the limiting one by definition \eqref{ln}.
\end{proof}\vspace*{0.05in}

The next result, which is a consequence of Theorem~\ref{COROLLARY:FUZZYEXPRESSION} and basic convex analysis, establishes the normal regularity of $M_\cap$ and gives us the {\em precise} formulas for calculating the normal cone and its relative interior under the normal regularity assumption imposed on $M(\omega)$ for almost all $\omega\in\O$.\vspace*{-0.05in}

\begin{corollary}[precise formulas for normals under normal regularity]\label{COROLLARY:RELATIVE_EINTERIOR} Let $M\colon\Omega\tto\R^n$ be a measurable multifunction with closed cone values. Assume that $M(\omega)$ is normally regular at the origin for a.e.\ $\omega\in\Omega$. Then the set $M_\cap$ is normally regular at the origin, and we have the equalities
\begin{align}
N(0;M_{\cap})&=\cl\left(\int_\Omega N\big(0;M(\omega)\big)d\mu(\omega)\right),\label{COROLLARY:RELATIVE_EINTERIOR01}\\
\ri\left(N(0;M_{\cap})\right)&=\ri\left(\int_\Omega N\big(0;M(\omega)\big)d\mu(\omega)\right).\label{COROLLARY:RELATIVE_EINTERIOR02}
\end{align}
\end{corollary}
\begin{proof} $\;$ Take $u^*\in\int_\Omega N(0;M(\omega))d\mu(\omega)$ and find integrable selection $x^*(\omega)\in N(0;M(\omega))$ a.e.\ with $u^*=\int_\Omega x^*(\omega)d\mu(\omega)$. It follows from the assumed normal regularity of $M(\omega)$ a.e.\ and the definition of $M_\cap$ that $u^*\in\Hat N(0;M_\cap)$, and so $\int_\Omega N(0;M(\omega))d\mu(\omega)\subset\Hat N(0;M_{\cap})$. Applying Theorem~\ref{COROLLARY:FUZZYEXPRESSION} yields
\begin{equation*}
\Hat N(0;M_{\cap})=N(0;M_{\cap})=\cl\left(\int_\Omega N\big(0;M(\omega)\big)d\mu(\omega)\right),
\end{equation*}
which verifies the normal regularity of $M_{\cap}$ at 0 together with \eqref{COROLLARY:RELATIVE_EINTERIOR01}. Since the set $\int_\Omega N\big(0;M(\omega)\big)d\mu(\omega)$ is convex, the relative interior formula \eqref{COROLLARY:RELATIVE_EINTERIOR02} follows from \eqref{COROLLARY:RELATIVE_EINTERIOR01} due to the classical fact of convex analysis.
\end{proof}\vspace*{0.08in}

Our further intention is to find verifiable conditions that allow us to drop the closure operation in the normal cone evaluations of type \eqref{COROLLARY:FUZZYEXPRESSION:EQUATION}. It is done below by using the extremal principle for measurable set-valued mappings established in Section~\ref{EXTREMALFORSETVALUED}. First we present the following lemma, which holds for general closed-valued measurable multifunctions.\vspace*{-0.05in}

\begin{lemma}[sequential optimality conditions for strict minimizers]\label{OPTCONDATZERO} Let $\bar{x}\in\R^n$ be a strict local minimizer of the optimization problem from Lemma~{\rm\ref{Fuzzyoptimalitycondition}}, and let $p\in(1,\infty)$. Then we have the alternative conditions:\\[1ex]
{\bf(i)} either there exist sequences of functions $x_k(\cdot)\in L^p(\O,\R^n)$ with $\|x_k(\cdot)-\ox\|_\infty\to 0$ and vectors $y_k\in\R^n$ with $\|y_k-\ox\|\to 0$ as $k\to\infty$ such that $0\in\Hat\partial h(y_k)+\int_{\Omega}\Hat N(x_k(\omega);M(\omega))d\mu(\omega)$, $k\in\N$;\\[1ex]
{\bf(ii)} or there exist sequences of functions $x_k(\cdot)$ and vectors $y_k$ as in {\rm(i)}, and also sequences of nonzero adjoint functions $x^*_k(\cdot)\in L^q(\O,\R^n)$ with $1/p+1/q=1$ and $x_k^*(\omega)\in\Hat N(x_k(\omega);M(\omega))$ for a.e.\ $\omega\in\O$ as well as vectors $u_k^*\in\Hat\partial^\infty h(y_k)$ such that $u_k^*+\int_{\Omega}x_k^\ast(\omega)d\mu(\omega)=0$ for all $k\in\N$.
\end{lemma}
\begin{proof} $\;$ Assume without loss of generality that the measure $\mu$ is finite and pick $\omega_0\notin\Omega$. Then we construct a new measure space $(\Tilde\O,\Tilde{\cal A},\Tilde\mu)$ exactly as in the proof of Lemma~\ref{Fuzzyoptimalitycondition}. Define further the measurable multifunction $\Tilde M\colon\O\tto\R^{n+1}$ on the new measure space by
\begin{equation*}
\Tilde M(\omega):=\left\{\begin{array}{cc}
\epi h&\text{ if }\;\omega=\omega_0,\\
M(\omega)\times(-\infty,h(\bar x)]&\text{ if }\;\omega\ne\omega_0
\end{array}\right.
\end{equation*}
and take a neighborhood $V$ of $\bar x$ on which this vector is a unique minimizer in the optimization problem under consideration.

Let us check that the mapping $\Tilde M$ is locally extremal at the origin in $L^p(\Tilde\O,\R^n)$ in the sense of Definition~\ref{local_extremal_at}. Indeed, denoting $U:=V\times\R$ and considering the measurable functions $a_k(\omega):=-(0,k^{-1}\1_{\{\omega=\omega_0\}}(\omega))$ a.e., where $\1_{A}(\omega)$ is the characteristic function of the set $A$, i.e., it equals to 1 on $A$ and 0 outside of $A$. We have that $\|a_k\|_p=k^{-1} \to 0$. Moreover, we can verify that $\bigcap_{\omega\in\Tilde{\Omega}\text{ a.e.}}(\Tilde M(\omega)-a_k(\omega))\cap U=\emp$, $k\in\N$, which justifies the claim.

Next we show that the nonoverlapping condition \eqref{nonoverlapping} holds for $\Tilde M$ with $U$ defined above. Take any $(x,\alpha)\in\Tilde M_{\cap}\cap U$ and observe that  $\alpha\ge h(x)$ due to \REVPP{   $(x,\alpha)\in  M(\omega_0)=\epi h$.} On the other hand, we have $x\in M_{\cap}\cap U$ and $\alpha\le h(\bar{x})$. Since $\ox$ is a strict minimizer of our problem, it implies that $(x,\alpha)=(\bar{x},h(\bar{x}))$, which justifies \eqref{nonoverlapping}.

Applying now Theorem~\ref{fuzzy_extremal_prin} gives us sequences $(x_k,\alpha_k)\in L^p(\O,\R^{n+1})$ and $(x_k^*,\alpha_k^*)\in L^q(\O,\R^{n+1})$ such that $\|(x_k,\alpha_k)\|_p\to 0$ as $k\to\infty$, $(x_k^*(\omega),\alpha_k^*(\omega))\in\Hat N((x_k(\omega),\alpha_k(\omega));\Tilde M(\omega))$ a.e.\ on $\O$, and
\begin{align}
(x_k^*(\omega_0),\alpha^\ast_k(\omega_0))+\int_{\Omega}\big(x_k^*(\omega),\alpha^\ast_k(\omega)\big)d\mu(\omega)&=0,\label{EP01}\\
\|x_k^*(\omega_0)\|^q+\|\alpha_k^*(\omega_0)\|^q+\int_{\Omega}\big(\|x_k^*(\omega)\|^q+\|\alpha_k^*(\omega)\|^q\big)d\mu(\omega)&=1\label{EP02}.
\end{align}
Note that estimate \eqref{est} of Theorem~\ref{fuzzy_extremal_prin} ensures actually the stronger convergence $\|(x_k,\alpha_k)\|_\infty\to 0$ as $k\to\infty$ due to the above choice of the sequence $\{a_k(\omega)\}$ in the extremality definition.

It follows from the constructions of $\Tilde M$ that $x_k^*(\omega)\in\Hat N(x_k(\omega);M(\omega))$ and $\alpha^*_k(\omega)\ge 0$ a.e.\ on $\Omega$. Thus \eqref{EP01} yields $\alpha^*_k(\omega_0)\le 0$ for all $k\in\N$. Furthermore
\begin{align*}
\Hat N\big((x_k(\omega_0),h(x_k(\omega_0));\Tilde M(\omega_0)\big)=\Hat N\big((x_k(\omega_0),h(x_k(\omega_0));\epi h\big),\quad k\in\N.
\end{align*}
Supposing that $\alpha_k^*(\omega_0)=0$ for infinitely many $k$ gives us $u^*_k:=x^*_k(\omega_0)\in\Hat\partial^\infty h(y_k)$ with $y_k:=x_k(\omega_0)\to 0$ as $k\to\infty$. Using \eqref{EP01} and \eqref{EP02} yields
\begin{align*}
u_k^*+\int_{\Omega}x_k^*(\omega)d\mu(\omega)=0\;\mbox{ and }\;\|u_k^*\|^q+\int_{\Omega}\|x_k^*(\omega)\|^qd\mu(\omega)=1
\end{align*}
for all $k\in\N$, which verifies assertion (ii) in this case. In the remaining case where $\alpha^*_k(\omega_0)<0$ for infinitely many $k$ we get $|\alpha_k^*(\omega_0)^{-1}|x_k^*(\omega_0)\in\Hat\partial h(y_k)$. Then \eqref{EP01} readily ensures the fulfillment of assertion (i).
\end{proof}\vspace*{0.05in}

To proceed further with dismissing the closure operation in the normal cone representations by employing Lemma~\ref{OPTCONDATZERO}, we need some qualification conditions for measurable multifunctions. Let us introduce two of them in the case of arbitrary closed-valued multifunctions.\vspace*{-0.05in}

\begin{definition}[normal qualification conditions]\label{definitionNQC01} Let $M\colon\Omega\tto\R^n$ be a measurable multifunction with closed values. We say that:\\[1ex]
{\bf(i)} The {\em regular normal qualification condition} holds for $M$ at $\ox\in M_\cap$ if there exists $\epsilon>0$ such that for all $x(\omega)\in\B(\bar{x},\epsilon)$ with a.e.\ $\omega\in\O$ we have
\begin{align}\label{NORMALQUALIFICATIONCOND01}
\left[\int_{\Omega}x^*(\omega)d\mu(\omega)=0,\;x^*(\omega)\in\Hat N\big(x(\omega);M(\omega)\big)\right]\Longrightarrow\big[x^*(\omega)=0\big].
\end{align}	
{\bf(ii)} The {\em limiting normal qualification condition} holds for $M$ at $\ox\in M_\cap$:
\begin{align*}
\left[\int_{\Omega}x^*(\omega)d\mu(\omega)=0,\;x^*(\omega)\in N\big(\ox;M(\omega)\big)\text{ a.e. }\right]\Longrightarrow\big[x^*(\omega)=0\;\text{ a.e. }\big].
\end{align*}
\end{definition}\vspace*{-0.01in}
Both qualification conditions of Definition~\ref{definitionNQC01} are new, while the limiting one is a natural extension of that in \cite[Definition~3.11]{mp1} and \cite[Definition~8.69]{m18} given for countably many sets, which extends in turn the standard normal qualification condition of variational analysis \cite{m06,m18,rw} for finite systems.

It is easy to see that the limiting qualification condition in Definition~\ref{definitionNQC01}(ii) implies the regular one in (i) if the set $\O$ is finite. It also happens when $M$ is {\em cone-valued} and $\ox=0$. Indeed, we can deduce the latter directly from the second inclusion of Proposition~\ref{inclusionpostivehomogeneous}.\vspace*{0.05in}

Let us present useful sufficient conditions for the validity of the limiting normal qualification condition from Definition~\ref{definitionNQC01}(ii) that also imply the one in \eqref{NORMALQUALIFICATIONCOND01} in the conic case of our main interest in this section.\vspace*{-0.05in}

\begin{proposition}[sufficient conditions for normal qualification]\label{suff-nc} Let $M\colon\Omega\tto\R^n$ be a measurable multifunction with closed values, let $\ox\in M_\cap$, and let $\bigcap_{\omega\in\Omega \text{ a.e.}}{\rm int}(M(\omega))\ne\emp$. Assume that either $M$ is convex-valued, or the sets $M(\omega)$ are cones which are normally regular at $\ox=0$ for a.e.\ $\omega\in\Omega$. Then $M(\cdot)$ satisfies the limiting normal qualification condition at $\ox$.
\end{proposition}\vspace*{-0.05in}
\begin{proof} $\;$ Considering the case of convex-valued mappings, take $x^*(\omega)\in N(\ox;M(\omega))$ with $\int_{\Omega}x^*(\omega)d\mu(\omega)=0$. Fix any $x\in\bigcap_{\omega\in\Omega \text{ a.e.}}{\rm int}(M(\omega))$ and $A\in\mathcal{A}$ and then get by the convexity of $M(\omega)$ that
\begin{align*}
0\ge\int_{A}\big\langle x^*(\omega),x-\ox\big\rangle d\mu(\omega)=-\int_{A^c}\big\langle x^*(\omega),x-\ox\big\rangle d\mu(\omega)\ge 0,
\end{align*}
where $A^c$ stands for the complement of $A$. Since $A\in\mathcal{A}$ was chosen arbitrarily, it shows that $\langle x^*(\omega),x-\ox\rangle=0$ for almost all $\omega\in\O$.

For any set $A\in\mathcal{A}$ with $\mu(A)=0$ and for any $\omega\in\Hat{\Omega}:=\Omega\backslash A$ with $x\in{\rm int}M(\omega)$ we get the relationships
\begin{equation*}
\langle x^*(\omega),x-\ox\rangle=0\;\mbox{ and }\;x^*(\omega)\in N\big(\ox;M(\omega)\big).
\end{equation*}
Furthermore, for each selected $\omega$ there exists a number $r_{\omega}>0$ such that $\mathbb{B}(x,r_{\omega})\subset M(\omega)$. Thus
$\langle x^*(\omega),h\rangle\le\langle x^*(\omega),\ox-x\rangle=0$ whenever $h\in\mathbb{B}(0,r_{\omega})$, which implies in turn that $x^*(\omega)=0$ for almost all $\omega\in\Omega$ and hence verifies the claimed normal regularity for convex-valued mappings. The proof for the  case of cone-valued multifunctions is similar.
\end{proof}\vspace*{0.05in}

Finally in this section, we derive desired representations of the regular normal cone to essential intersections and its interior without using the closure operation. The first part of this theorem holds for general measurable mappings, while the second one addresses cone-valued multifunctions.\vspace*{-0.05in}

\begin{theorem}[normal cone formulas without closure]\label{THEOREM:EXACTINTEGRALREPRESENTATION01} Let $M\colon\Omega\tto\R^n$ be a measurable multifunction with closed values satisfying the regular normal qualification condition \eqref{NORMALQUALIFICATIONCOND01} at $\ox\in M_\cap$. Then the following hold:\\[1ex]
{\bf(i)} Take any $x^*\in\Hat N(\ox;M_\cap)$ for which there is $\epsilon>0$ with
\begin{equation}\label{str}
\langle x^*,x-\bar{x}\rangle<0\;\mbox{ whenever }\;x\in M_\cap\cap\mathbb{B}(\bar{x},\epsilon)\backslash\{\bar{x}\}.
\end{equation}
Then there exists a measurable selection $x(\omega)\in M(\omega)\cap\mathbb{B}(\bar{x},\epsilon)$ such that
\begin{align}\label{InclusionFrechet}
x^*\in\int_{\Omega}\Hat N\big(x(\omega);M(\omega)\big)d\mu(\omega).
\end{align}
{\bf(ii)} If the values of $M(\cdot)$ are closed cones, then we have
\begin{align}\label{Corollary_eq02}
{\rm int}\Hat N(0;M_{\cap})\subset\int_{\Omega}N\big(0;M(\omega)\big)d\mu(\omega).
\end{align}
\end{theorem}\vspace*{-0.05in}
\begin{proof} $\;$ To verify (i), observe that for the vector $x^*$ satisfying the assumptions therein we get that the function $h(x):=-\la x^*,x\ra$ attains its strict local minimum at $\ox$ subject to the constraints $x\in M(\omega)$ a.e.\ on $\O$. Then Lemma~\ref{OPTCONDATZERO} gives us the two alternative conditions. It is easy to see that the second among them is ruled out by the imposed regular normal qualification condition \eqref{NORMALQUALIFICATIONCOND01}. Thus we arrive at the necessary condition in Lemma~\ref{OPTCONDATZERO}(i), which reduces to inclusion \eqref{InclusionFrechet} in the case of the selected function $h(x)$.

To verify now assertion (ii), we show first that the inclusion
\begin{align}\label{Corollary_eq01}
{\rm int}\Hat N(0;M_{\cap})\subset\bigcap_{\epsilon>0}\bigcup_{x\in L^\infty(\O,\B(0,\epsilon))}\int_{\Omega} \Hat N\big(x(\omega);M(\omega)\big)d\mu(\omega)
\end{align}
holds for any closed cone-valued measurable multifunction. To proceed, pick any $\epsilon>0$ and $x^*\in{\rm int}\Hat N(0;M_{\cap})$, and then find $r>0$ such that $\mathbb{B}(x^*,r)\subset\Hat N(0;M_{\cap})$. Fixing $x\in M_\cap\backslash\{\bar{x}\}$ and defining $u^*:=r\frac{x}{\|x\|}$, we get from the above constructions that $x^*+u^*\in\Hat N(0;M_{\cap})$ and therefore
\begin{equation*}
\langle x^*,x\rangle=\langle x^*+u^*,x\rangle-\langle u^*,x\rangle\le-r\|x\|<0.
\end{equation*}
It tells us that \eqref{str} holds, which yields \eqref{InclusionFrechet} with some measurable selection $x(\omega)\in \B(\bar{x},\epsilon)$ a.e.\ on $\Omega$ by assertion (i) established above. This clearly justifies \eqref{Corollary_eq01}. Furthermore, it follows from Proposition~\ref{inclusionpostivehomogeneous} that
\begin{equation}\label{eqnew}
\Hat N\big(x(\omega);M(\omega)\big)\subset N\big(0;M(\omega)\big)\mbox{\text{ a.e.  on } }\;\omega\in\O
\end{equation}
due to the cone-valuedness assumption on $M(\cdot)$. Hence \eqref{eqnew} implies that the right-hand side of \eqref{Corollary_eq01} is included in the right-hand side of \eqref{Corollary_eq02}, and thus we complete the proof of the theorem.
\end{proof}\vspace*{0.05in}	

Note that assertion (ii) of Theorem~\ref{THEOREM:EXACTINTEGRALREPRESENTATION01} is a counterpart of formula \eqref{COROLLARY:RELATIVE_EINTERIOR02} in Corollary~\ref{COROLLARY:RELATIVE_EINTERIOR} obtained without imposing {\em any regularity condition}.\vspace*{-0.17in}

\section{Normals to essential intersections via CHIP}\label{chip}\vspace*{-0.1in}

In this section we extend the major normal cone formulas obtained in Section~\ref{SectionIntegralRepresentation} for cone-valued multifunctions to a general class of closed-valued multifunctions on measure spaces. To furnish this, we first introduce and investigate the so-called {\em CHIP} (conical hull intersection property) for measurable multifunctions, which has been studied in the literature under this name for the classes of finitely many convex sets (see, e.g., \cite{bbl,et} and the references therein) as well as countably many convex \cite{lnp} and nonconvex \cite{m18,mp2} sets.\vspace*{-0.05in}

\begin{definition}[CHIP for measurable multifunctions]\label{CHIPFORMEASURABLESET} Let $M\colon\Omega\tto\R^n$ be a measurable multifunction on the measure space $(\Omega,\mathcal{A},\mu)$. We say that the {\em measurable CHIP} ({\em conical hull intersection property}) holds for $M(\cdot)$ with respect to $(\Omega,\mathcal{A},\mu)$ at $\ox\in M_\cap$ if
\begin{align}\label{equation_09}
T(\ox;M_{\cap})=\bigcap_{\omega\in\Omega \text{ a.e.}}T\big(\ox;M({\omega})\big).
\end{align}
When no confusion arises about the measure space, we simple say that the {\em measurable CHIP} holds for $M(\cdot)$ at $\ox$.
\end{definition}\vspace*{-0.07in}

It is important to mention that CHIP holds automatically for multifunctions with closed cone values at the origin.

Let us present some sufficient conditions for the fulfillment of CHIP. The following new property postulates a certain uniformity over the set of tangential directions. Having in mind applications to semi-infinite programming in Section~\ref{AplSIP}, we consider below arbitrary index sets, not just measure spaces.\vspace*{-0.05in}

\begin{definition}[tangential stability]\label{stab} We say that a set $\Th\subset\R^n$ is {\em tangentially stable} at $\bar{x}\in\Th$ with respect to some $U\subset\R^n$ if
\begin{align}\label{stabeq}
T(\bar{x};\Th)\cap U\subset(\Th-\bar{x}).
\end{align}
A family of sets $\{\Th_t\}_{t\in\mathcal{T}}\subset\R^n$ is {\em uniformly tangentially stable} at a common point $\bar{x}$ if there exists an (open) neighborhood $U$ of zero such that the sets $\Th_t$ are tangentially stable at $\bar{x}$ with respect to $U$ for all $t\in\mathcal{T}$. In the case where $\mathcal{T}$ is a measure space, $\{\Th_t\}_{t\in\mathcal{T} }$ is (uniformly) {\em almost everywhere tangentially stable} at $\bar{x}$ provided that the previous property holds for almost all $t\in\mathcal{T}$.
\end{definition}\vspace*{-0.03in}

Note that the tangential stability property \eqref{stabeq} holds for $\Th$ at $\ox$ if either $\ox\in{\rm int}\Th$, or $\Th$ is a cone with $\ox=0$, or $\Th$ is the complement of an open convex set. The next lemma establishes the validity of CHIP under the uniform tangential stability of set systems.\vspace*{-0.05in}

\begin{lemma}[tangential stability implies CHIP]\label{lemmaCHIP} Consider a family of closed sets $\{\Th_t\}_{t\in\mathcal{T}}$ with $\bar{x}\in\bigcap_{t\in\mathcal{T}}\Th_t$ and assume that the system $\{\Th_t\}_{t\in\mathcal{T}}$ is uniformly tangentially stable at $\bar{x}$. Then we have
\begin{align}\label{stab1}
T\Big(\bar{x};\bigcap_{t\in\mathcal{T}}\Th_t\Big)=\bigcap_{t\in\mathcal{T}}T(\bar{x};\Th_t).
\end{align}
If ${\cal T}$ is a measure space and the family $\{\Th_t\}_{t\in\mathcal{T}}$ is almost everywhere tangentially stable at $\ox$, then \eqref{stab1} holds with $t\in{\cal T}$ a.e.\ therein.
\end{lemma}\vspace*{-0.05in}
\begin{proof} $\;$ It is sufficient to verify the nontrivial inclusion ``$\supset$" in \eqref{stab1}. By the assumed uniform tangential stability of $\{\Th_t\}$ we can take an open neighborhood $U$ of zero such that \eqref{stabeq} holds. It clearly yields, by taking into account that the set $\bigcap_{t\in\mathcal{T}}T(\bar{x};\Th_t)$ is a cone, the relationships
\begin{equation*}
\begin{aligned}
\bigcap_{t\in\mathcal{T}}T(\bar{x};\Th_t)&=T\left(0;\bigcap_{t\in\mathcal{T}}T(\bar{x};\Th_t)\right)=T\left(0;\left(\bigcap_{t\in\mathcal{T}}T(\bar{x};\Th_t)\right)\bigcap U\right)\\&\subset T\left(0;\bigcap_{t\in\mathcal{T}}\left(\Th_t-\bar{x}\right)\right)=T\left(\bar{x};\bigcap_{t\in\mathcal{T}}\Th_t\right),
\end{aligned}
\end{equation*}
which ensure in turn the claimed CHIP of the family $\{\Th_t\}_{t\in{\cal T}}$.
\end{proof}\vspace*{0.05in}

The following consequence of Lemma~\ref{lemmaCHIP} is useful for applications to optimization problems with inequality constraints.\vspace*{-0.05in}

\begin{corollary}[CHIP for infinite inequality systems]\label{PropositionCHIP} Let $\Th_t:=\{x\in\mathbb{R}^n|\;f(t,x)\le 0\}$ with an arbitrary index set ${\cal T}$, where $f(t,x):=\la a(t),x\ra-b(t)$ with $a\colon\mathcal{T}\to\mathbb{R}^n$ and $b\colon\mathcal{T}\to\R$. Taking $\ox\in\R^n$ and the collection of active indexes ${\cal T}_f:=\{t\in\mathcal{T}|\;f(t,\ox)=0\}$, suppose that $\bar{x}\in{\rm int}\bigcap_{t\in T\backslash\mathcal{T}_f}\{x\in\R^n|\;f(t,x)<0\}$. Then we have that CHIP \eqref{stab1} holds for $\{\Th_t\}_{t\in{\cal T}}$ at $\ox$. If ${\cal T}$ is a measure space with $a(\cdot)$ and $b(\cdot)$ being measurable on it, then the measurable CHIP is satisfied with $t\in{\cal T}$ a.e.\ therein.
\end{corollary}\vspace*{-0.05in}
\begin{proof} $\;$ To verify the CHIP \eqref{stab1}, it is sufficient to show by Lemma~\ref{lemmaCHIP} that the system $(\Th_t)_{t\in\mathcal{T}}$ is uniformly tangentially stable at $\bar{x}$. Indeed,
consider the open set $U:={\rm int}\left(\bigcap_{t\in T\backslash\mathcal{T}_f}\{x\in\mathbb{R}^n|\;f(t,x)<0\}-\bar{x}\right)$ and observe that for every $t\in\mathcal{T}_f$ the set $\Th_t-\bar{x}$ is a cone. This implies that $T(\bar{x};\Th_t)=\Th_t-\bar{x}$. Furthermore, for all $t\notin \mathcal{T}_f$ we have that $\ox$ is an interior point of $\Th_t$, and hence 
$T(\bar{x};\Th_t)=\mathbb{R}^n$. It tells us therefore that $T(\bar{x};\Th_t)\cap U\subset U\subset\Th_t-\bar{x}$.
\end{proof}\vspace*{0.05in}

The next theorem presents major counterparts for general measurable multifunctions of the normal cone formulas from Theorem~\ref{COROLLARY:FUZZYEXPRESSION} and Corollary~\ref{COROLLARY:RELATIVE_EINTERIOR} obtained above for cone-valued multifunctions.\vspace*{-0.05in}

\begin{theorem}[normal cone evaluations for measurable multifunctions]\label{FrechetnormalunderCHIP} Let $M\colon\Omega\tto\mathbb{R}^n$ be a measurable multifunction with closed values, and let $\ox\in M_{\cap}$ for its essential intersection \eqref{ess-inter}. Assume that the measurable CHIP holds for $M(\cdot)$ at $\ox$. Then we have the upper estimate
\begin{align}\label{FrechetnormalunderCHIPequation}
\Hat N(\ox;M_\cap)\subset\cl\left(\int_\Omega N\big(\ox;M(\omega)\big)d\mu(\omega)\right).
\end{align}
If in addition $M(\omega)$ is normally regular at $\ox$ for almost all $\omega\in\Omega$, then inclusion \eqref{FrechetnormalunderCHIPequation} holds as equality, and we also get
\begin{align}\label{FrechetnormalunderCHIPequation2}
\ri\left(\Hat N(\ox;M_\cap)\right)=\ri\left(\int_\Omega N\big(\ox;M(\omega)\big)d\mu(\omega)\right).
\end{align}
\end{theorem}\vspace*{-0.05in}
\begin{proof} $\;$ It follows from the definitions that $\Hat N(\ox;M_\cap)=\Hat N(0;T(\ox;M(\omega)))$. Furthermore, the imposed CHIP yields the fulfillment of \eqref{equation_09}. Hence, by the duality between $T(\cdot;\Th)$ and $\Hat N(\cdot;\Th)$, we get
\begin{align}\label{equation_10}
\Hat N\big(\ox;M_\cap\big)=\Hat N\big(0;T(\ox;M_\cap)\big)=\Hat N\Big(0;\bigcap_{\omega\in\Omega\;a.e}T\big(\ox;M(\omega)\big)\Big).
\end{align}
Then applying Theorem~\ref{COROLLARY:FUZZYEXPRESSION} to the cones in \eqref{equation_10} gives us the inclusion
\begin{equation*}
\Hat N\Big(0;\bigcap_{\omega\in\Omega \text{ a.e. }}  T\big(\ox;M(\omega)\big)\Big)\subset\cl\left(\int_\Omega N\big(0;T(\ox;M(\omega))\big)d\mu(\omega)\right),
\end{equation*}
which yields \eqref{FrechetnormalunderCHIPequation} due to $N(0;T(\ox;M(\omega)))\subset N(\ox;M(\omega))$ for all $\omega\in\Omega$.\vspace*{0.02in}

Proceeding now in the case where $M(\cdot)$ is normally regular at $\ox$, we can easily observe that the tangent cone $T(\ox;M(\omega))$ is also normally regular at the origin for almost all $\omega\in\O$. Applying then Corollary~\ref{COROLLARY:RELATIVE_EINTERIOR} together with \eqref{equation_09} and \eqref{equation_10} tells us that
\begin{align*}
\Hat N\big(\ox;M_\cap)&=\Hat N\big(0;T(\ox;M_\cap)\big)=\cl\Bigg(\int_\Omega\Hat N\big(0;T(\ox; M(\omega))\big)d\mu(\omega)\Bigg)\\&=\cl\Bigg(\int_\Omega\Hat N\big(\ox;M(\omega)\big)d\mu(\omega)\Bigg),
\end{align*}
which justifies the equality in \eqref{FrechetnormalunderCHIPequation}. The relative interior formula \eqref{FrechetnormalunderCHIPequation2} is verified similarly by employing Corollary~\ref{COROLLARY:RELATIVE_EINTERIOR}.\end{proof}\vspace*{0.05in}

Finally, we use CHIP and the conic result of Theorem~\ref{THEOREM:EXACTINTEGRALREPRESENTATION01}(ii) to dismiss the closure operation in the normal cone estimate for general multifunctions. \vspace*{-0.05in}

\begin{proposition}[normal cone estimate without closure]\label{FrechetnormalunderCHIP_int} Let $M\colon\Omega\tto\mathbb{R}^n$ be a measurable multifunction with closed values, and let $\ox\in M_{\cap}$. Suppose that the measurable CHIP holds for $M(\cdot)$ at $\ox$. Then we have the interior estimate
\begin{align*}
{\rm int}\Hat N(\ox;M_\cap)\subset\int_\Omega N\big(\ox;M(\omega)\big)d\mu(\omega).
\end{align*}	
\end{proposition}\vspace*{-0.1in}
\begin{proof} $\;$ First $\Hat N(\ox;M_\cap)=\Hat N(0;T(\ox;M_\cap))$. Then using the assumed CHIP and Theorem~\ref{THEOREM:EXACTINTEGRALREPRESENTATION01}(ii) for cone-valued mappings gives us
\begin{equation*}
{\rm int}\Hat N(0;M_\cap)\subset\int_{\Omega}N\big(0; T(\ox; M(\omega))   \big)d\mu(\omega)\subset\int_{\Omega}N\big(\ox;M(\omega)\big)d\mu(\omega),
\end{equation*}
which verifies the claimed upper estimate.
\end{proof}\vspace*{-0.2in}

\section{Applications to Stochastic Programming}\label{ApplicationSTOCHASTIC}\vspace*{-0.1in}

In this section we consider the {\em stochastic optimization problem} with the {\em random} constraint sets formulated in \eqref{firstproblem}. Suppose in what follows that $M\colon\O\to\R^n$ is a measurable multifunction with closed values and that $h\colon\R^n\to\oR$ is an l.s.c.\ function around the reference point.

Based on the normal cone formulas for the essential intersection of $M$ obtained in Sections~\ref{SectionIntegralRepresentation} and \ref{chip}, we are now ready to derive new necessary optimality conditions for local minimizers of \eqref{firstproblem}.\vspace*{-0.05in}

\begin{theorem}[necessary optimality conditions for general stochastic programs]\label{Theorem1:Applications} For a local minimizer $\ox\in M_\cap$ of \eqref{firstproblem} the following hold:\\[1ex]
{\bf(i)} Assume that $M_\cap$ is normally regular at $\ox$, that $M(\cdot)$ satisfies the measurable CHIP at $\ox$, and that the qualification condition
\begin{align}\label{normalcondition}
\big(-\subm^\infty h(\bar x)\big)\cap\cl\Big(\int_{\Omega}N\big(\bar{x};M(\omega)\big)d\mu(\omega)\Big)=\{0\}
\end{align}
is satisfied. Then we have the necessary optimality condition
\begin{align}\label{Theorem1:Applications:eq1}
0\in\subm h(\bar{x})+\cl\Big(\int_{\Omega}N\big(\bar{x};M(\omega)\big)d\mu(\omega)\Big).
\end{align}
{\bf(ii)} Assume that the sets $M(\omega)$ are cones for a.e.\ $\omega\in\O$ and that $\ox=0$. Then the fulfilment of \eqref{normalcondition} ensures that \eqref{Theorem1:Applications:eq1} holds at this point.
\end{theorem}\vspace*{-0.05in}
\begin{proof} $\;$ It follows from basic variational analysis (see, e.g., \cite[Proposition~5.3]{m06}) that $0\in\subm h(\bar{x})+N(\bar{x};M_{\cap})$ if $(-\subm^\infty h(\bar{x}))\cap N(\bar{x};M_{\cap})=\{0\}$. Employing now Theorem~\ref{FrechetnormalunderCHIP} under the assumptions made in (i), we get
\begin{equation*}
N(\ox;M_\cap)=\Hat N(\bar{x};M_{\cap})\subset\cl\Big(\int_{\Omega}N\big(\bar{x};M(\omega)\big)d\mu(\omega)\Big),
\end{equation*}
which clearly ensures the fulfillment of condition \eqref{Theorem1:Applications:eq1} if \eqref{normalcondition} holds.

To verify (ii) in the case of cone values, we proceed similarly by using the second statement of Theorem~\ref{COROLLARY:FUZZYEXPRESSION} instead of Theorem~\ref{FrechetnormalunderCHIP}. Note that in this case neither CHIP nor normal regularity assumptions are needed.
\end{proof}\vspace*{0.05in}

Let us discuss the obtained estimates and their consequences.\vspace*{-0.05in}

\begin{remark} $(${\em discussions on optimality conditions}$)$\label{diss} Observe the following:\\[1ex]
{\bf(a)} The qualification condition \eqref{normalcondition} holds {\em automatically} if $h$ is locally {\em Lipschitzian} around $\ox$ due the characterization $\partial^\infty h(\ox)=\{0\}$ of this property.\\[1ex]
{\bf(b)} The result of Theorem~\ref{Theorem1:Applications}(i) clearly yields those from \cite[Theorem~4.2]{mp2} and \cite[Theorem~8.77]{m18} for problem \eqref{firstproblem} with countably many geometric constraints. Moreover, we significantly extend the previous developments in the case of countable constraints by {\em dropping the normal qualification condition} from Definition~\ref{definitionNQC01}(ii) imposed in \cite{m18,mp2}. The result of Theorem~\ref{Theorem1:Applications}(ii) for cone-valued multifunctions without the normal regularity has never been observed earlier even for countably many constraints.\\[1ex]
{\bf(c)} The result of Theorem~\ref{THEOREM:EXACTINTEGRALREPRESENTATION01}(i) allows us, under the additional assumption therein, to avoid the closure operation in both assertions of Theorem~\ref{Theorem1:Applications}.\\[1ex]
{\bf(d)} Finally, let us mention that the normal regularity of $M_\cap$ and the normal qualification condition \eqref{normalcondition} can  be replaced by the  assumption that $h$ is Fr\'echet differentiable at $\bar x$. Indeed, we get it by using \cite[Proposition~1.107]{m06} instead of \cite[Proposition~5.3]{m06} to derive the necessary optimality conditions.
\end{remark}\vspace*{-0.07in}

We conclude this section by specifying the results of Theorem~\ref{Theorem1:Applications} to the class of stochastic optimization problems with {\em random inequality} constraints
\begin{align}\label{lambdaintermsoff}
M(\omega):=\big\{x\in\mathbb{R}^n\big|\;f(\omega,x)\le 0\big\},
\end{align}
where $f\colon\Omega\times\mathbb{R}^n\to\oR$ is a {\em normal integrand} such that $f_\omega(\cdot):=f(\omega,\cdot)$ is locally Lipschitzian around $\bar{x}$ for almost all $\omega\in\Omega$.\vspace*{-0.05in}

\begin{corollary}[necessary optimality conditions for stochastic programs with inequality constraints]\label{Theorem01:inequalitysystem} Let $\ox$ be a local minimizer of problem \eqref{firstproblem} with the inequality constraints \eqref{lambdaintermsoff}. The following assertions hold:\\[1ex]
{\bf(i)} In addition to the normal regularity and CHIP assumptions of Theorem~{\rm\ref{Theorem1:Applications}}, impose the qualification conditions
\begin{align}\label{qc}
0\notin\sub f_\omega(\bar{x})\;\text{ for almost all }\;\omega\in\Omega_{f}:=\big\{\omega\in\Omega\big|\;f(\omega,\bar{x})=0\big\},
\end{align}
\begin{align}\label{normalconditionTheorem01:inequalitysystem}
\big(-\subm^\infty h(\bar{x})\big)\cap\cl\Big(\int_{\Omega_f}{\rm cone}\big(\subm f_\omega(\ox)\big)d\mu(\omega)\Big)=\{0\}.
\end{align}
Then we have the necessary optimality condition
\begin{align}\label{Theorem01:inequalitysystem:conclusion}
0\in\subm h(\ox)+\cl\Big(\int_{\Omega_f}{\rm cone}\big(\subm f_\omega(\ox)\big)d\mu(\omega)\Big).
\end{align}
{\bf(ii)} Assume that $f(\omega,\lm x)\le\lm f(\omega,x)$ for all $\lm\ge 0$, all $x\in\R^n$, and a.e.\ $\omega\in\O$. Then the fulfillment of \eqref{normalconditionTheorem01:inequalitysystem} ensures that \eqref{Theorem01:inequalitysystem:conclusion} holds at $\ox=0$.
\end{corollary}\vspace*{-0.04in}
\begin{proof} $\;$ It follows from \eqref{qc} by \cite[Proposition~10.3]{rw} that
\begin{equation*}
N\big(\bar{x};M(\omega)\big)\subset\R_{+}\subm f_\omega(\bar{x})\;\text{ for almost all }\omega\in\Omega\;\text{ with }\;f_\omega(\bar{x})=0.
\end{equation*}
If $f_\omega(\bar{x})<0$, we get by the continuity of $f_\omega$ that $N(\bar{x};M(\omega))=\{0\}$. Thus
\begin{equation*}
\int_{\Omega}N\big(\bar{x};M(\omega)\big)d\mu(\omega)\subset\cl\Big(\int_{\Omega_f}{\rm cone}\big(\subm f_\omega(x)\big)d\mu(\omega)\Big),
\end{equation*}
which allows us to deduce assertion (i) from Theorem~\ref{Theorem1:Applications}(i). To verify assertion (ii), it is sufficient to observe that the additional assumption therein ensures that the sets $M(\omega)$ in \eqref{lambdaintermsoff} are cones, and then apply Theorem~\ref{Theorem1:Applications}(ii).
\end{proof}\vspace*{0.05in}

Note  that the normal regularity assumption on the mapping \eqref{lambdaintermsoff} can be replaced by the {\em subdifferential/lower regularity} of $f_\omega$ at $\ox$ (see \cite{m06,rw}) and that the sufficient conditions for the CHIP assumption for inequality constraints are given in Proposition~\ref{PropositionCHIP}.\vspace*{-0.2in}

\section{Applications to semi-infinite programming}\label{AplSIP}\vspace*{-0.1in}
	
The concluding section of the paper is devoted to applications of the results obtained above to general problems of {\em semi-infinite programming} given by
\begin{align}\label{IOPTIMIZATION.FUZZY.SIP}
\mbox{minimize }\;h(x)\;\mbox{ subject to }\;x\in M(t)\;\mbox{ for all }\;t\in{\cal T},
\end{align}
where $M\colon{\cal T}\tto\mathbb{R}^n$ is a multifunction with closed values, and where the index set ${\cal T}$ is a metric space. The conventional setting of \eqref{IOPTIMIZATION.FUZZY.SIP} concerns linear and convex problems with inequality constraints defined on compact sets ${\cal T}$, while more recently various classes of semi-infinite programs with inequality constraints on noncompact sets have been also under consideration; see, e.g., \cite{clmp2,gl,lnp,m18,mp2} and the references therein. Problems of type \eqref{IOPTIMIZATION.FUZZY.SIP} with countable set constraints were studied in \cite{m18,mp2}.

Note that, in contrast to problem \eqref{firstproblem} from the previous section, program \eqref{IOPTIMIZATION.FUZZY.SIP} does not explicitly contain any measure. However, we can associate with \eqref{IOPTIMIZATION.FUZZY.SIP} a measure space constructed as follows. For a closed set $A\subset \mathcal{T}$, let $\mathcal{B}(A)$ be the Borel $\sigma$-algebra on $A$. We say that a measure on $\mathcal{B}(A)$ is {\em strictly positive} if every nonempty open subset of $A$ has strictly positive measure and then denote by $\mathfrak{M}_+(A)$ the set of all the finite strictly positive measures on $\mathcal{B}(A)$. For simplicity, we confine ourselves to the case where $M$ is an {\em outer semicontinuous}, i.e., $\Limsup_{s\to t}M(s)\subset M(t)$ for all $t\in{\cal T}$.\vspace*{0.03in}

The next theorem presents general necessary optimality conditions for nonsmooth and nonconvex semi-infinite programs of type \eqref{IOPTIMIZATION.FUZZY.SIP} with infinitely many set constraints indexed via arbitrary metric spaces.\vspace*{-0.05in}

\begin{theorem}[necessary optimality conditions for semi-infinite programs with set constraints]\label{firstresultCHIP} Let $\ox$ be a local minimizer of problem \eqref{IOPTIMIZATION.FUZZY.SIP}, where the cost function $h(\cdot)$ is locally Lipschitzian around $\ox$.\\[1ex]
{\bf(i)} Assume that the set $\bigcap_{t\in \cal T}M(t)$ is normally regular at $\ox$ and that for each dense set $A\subset{\cal T}$ the CHIP condition
\begin{align}\label{CHIPDENSESETS}
T\Big(\bar{x};\bigcap_{t\in A}M(t)\Big)=\bigcap_{t\in A}T\big(\bar{x};M(t)\big)
\end{align}
is satisfied. Then for every measure $\nu\in\mathfrak{M}_+({\cal T})$ we have
\begin{align}\label{Inclusion}
0\in\sub h\big(\bar{x}\big)+{\rm cl}\Big(\int_{{\cal T}}N\big(\bar{x};M(t)\big)d\nu(t)\Big).
\end{align}
{\bf(ii)} Assume that the set $M(t)$ is a cone for each $t\in{\cal T}$, and that $\ox =0$. Then the optimality condition \eqref{Theorem1:Applications:eq1} holds at the origin.
\end{theorem}\vspace*{-0.05in}
\begin{proof} $\;$ Denote by $({\cal T},\mathcal{A},\mu)$ the completion of $({\cal T},\mathcal{B}({\cal T}),\nu)$ and consider the following optimization problem of type \eqref{firstproblem} from Section~\ref{ApplicationSTOCHASTIC}:
\begin{align}\label{IOPTIMIZATION.FUZZY.SIPmu}
\begin{array}{c}
\mbox{minimize }\;h(x)\;\mbox{ subject to }\;x\in M_{\cap}.
\end{array}
\end{align}
Let us check that $\bar{x}$ is a local minimizer of \eqref{IOPTIMIZATION.FUZZY.SIPmu} and that all the assumptions of Theorem~\ref{Theorem1:Applications} are satisfied for  \eqref{IOPTIMIZATION.FUZZY.SIPmu}. Indeed, the imposed outer semicontinuity of $M(\cdot)$ ensures that the distance function $t\mapsto d_{M(t)}(x)$ is l.s.c.\ on ${\cal T}$ for all $x\in\mathbb{R}^n$, which yields the measurability of $M(\cdot)$ with respect to $({\cal T},\mathcal{A},\mu)$ by employing \cite[Theorems~14.2 and 14.8]{rw}.

It is easy to see that $M_{\cap}$ contains the feasible set of \eqref{IOPTIMIZATION.FUZZY.SIP}. Furthermore, if $x\in M_{\cap}$, then there exists a set $A\in\mathcal{A}$ such that $\mu({\cal T}\backslash A)=0$ and $x\in M(s)$ for all $s\in A$. Observe that $A$ is dense on ${\cal T}$; otherwise there exists an open set $U$ such that $A\cap U\ne\emp$, which contradicts the strict positivity of $\mu$. In particular, for every $t\in{\cal T}$ there exists a sequence $t_k\to t$ as $k\to\infty$ with $x\in M(t_k)$ for all $k\in\N$, and then the outer semicontinuity of $M(\cdot)$ yields $x\in M(t)$. This shows that $x\in M(t)$ for all ${\cal T}$ if and only if $x\in M_{\cap}$. It follows also from the above arguments that $\bar{x}$ is a local minimizer of \eqref{IOPTIMIZATION.FUZZY.SIPmu}.

To verify now assertion (i) of the theorem, we are going to apply Theorem~\ref{Theorem1:Applications}(i) to problem \eqref{IOPTIMIZATION.FUZZY.SIPmu}. As follows from the above, $M_\cap$ is normally regular at $\bar{x}$. Furthermore, the qualification condition \eqref{normalcondition} holds due the imposed Lipschitz continuity of $h(\cdot)$ around $\ox$; see Remark~\ref{diss}(a). To apply the result of Theorem~\ref{Theorem1:Applications}(i), it remains to show that the assumed CHIP \eqref{CHIPDENSESETS} yields the validity of the measurable CHIP for $M_\cap$ at $\ox$ with respect to $\mu$. To proceed, pick any $v\in\bigcap_{t\in{\cal T}\;\mu\text{-a.e.}}T(\bar{x};M(t))$ and find a dense set $A\subset{\cal T}$ of full measure such that $v\in\bigcap_{t\in A}T(\bar{x};M(t))$. Hence  \eqref{CHIPDENSESETS} tells us that $v\in T(\ox;\bigcap_{t\in A}M(t))$. It follows from the outer semicontinuity of $M(\cdot)$ that
\begin{align*}
\bigcap_{t\in A}M(t)=\bigcap_{t\in{\cal T}}M(t),\;\mbox{ and so }\;\bigcap_{t\in{\cal T}\;\mu\text{-a.e.}}T\big(\ox;M(t)\big)\subset T(\ox;M_\cap),
\end{align*}
which justifies the measurable CHIP for $M_\cap$ at $\ox$. Using finally the necessary optimality condition \eqref{Theorem1:Applications:eq1} of Theorem~\ref{Theorem1:Applications}(i), we arrive at \eqref{Inclusion} and thus complete the proof of assertion (i). Assertion (ii) of the theorem follows directly from Theorem~\ref{Theorem1:Applications}(ii) and the arguments above.
\end{proof}\vspace*{-0.05in}

\begin{remark}$(${\em Fr\'echet differentiable costs}$)$\label{remark7.2} It easily follows from the proof of Theorem~\ref{firstresultCHIP} that the normal regularity of $\bigcap_{t\in\cal T}M(t)$ at $\ox$ therein is not needed if the cost function $h$ is Fr\'echet differentiability at $\ox$; see Remark~\ref{diss}(d).
\end{remark}\vspace*{-0.05in}

Next we consider semi-infinite programs with {\em inequality constraints}:
\begin{align}\label{IOPTIMIZATION.FUZZY.SIP.3}
\begin{array}{c}
\mbox{minimize }\;h(x)\;\mbox{ subject to }\;x \in M(t):=\big\{x\big|\;f(t,x)\le 0\big\},\;t\in{\cal T},
\end{array}
\end{align}
where $h\colon\R^n\to\R$ is continuously differentiable while $f\colon{\cal T}\times\R^n\to\R$ is continuous with respect to $t$ and continuously differentiable with respect to $x$.\vspace*{-0.08in}

\begin{theorem}[optimality conditions for semi-infinite programs with inequality constraints]\label{theoremFinal} Let $\bar{x}$ be a local minimizer of \eqref{IOPTIMIZATION.FUZZY.SIP.3} such that
\begin{equation*}
\ox\in{\rm int}\Big(\bigcap_{t\in{\cal T}\backslash{\cal T}_f}\big\{x\in\mathbb{R}^n\big|\;f(t,x)<0\big\}\Big)\;\mbox{ with }\;{\cal T}_f:=\big\{t\in{\cal T}\big|\;f(t,\bar{x})=0\big\},
\end{equation*}
that $\nabla_x f(t,\ox)\ne 0$ for all $t\in{\cal T}_f$, and that the mapping $t\mapsto\nabla_x f(t,\ox)$ is continuous on ${\cal T}_f$. Furthermore, suppose that the CHIP assumption \eqref{CHIPDENSESETS} is satisfied at $\ox$ with $A:={\cal T}_f$ therein. Then we have
\begin{align}\label{inclusiongradient}
0\in\nabla h(\bar{x})+{\rm cl}\Big(\int_{{\cal T}_f}{\rm cone}\big\{\nabla_x f(t,\bar{x})\big\}d\nu(t)\Big)\;\mbox{ for any }\;\nu\in\mathfrak{M}_{+}({\cal T}_f).
\end{align}
\end{theorem}\vspace*{-0.07in}
\begin{proof} $\;$ Without loss of generality, from now on we consider problem \eqref{IOPTIMIZATION.FUZZY.SIP.3} for $t\in{\cal T}_f$. Applying to this problem the results of Theorem~\ref{firstresultCHIP} and Remark~\ref{remark7.2}), observe that all the corresponding assumptions can be easily verified with the exception  of CHIP \eqref{CHIPDENSESETS} over the set ${\cal T}_f$. Thus it remains to check that the imposed CHIP assumption yields the fulfillment of CHIP \eqref{CHIPDENSESETS} for any dense subset $A\subset{\cal T}_f$. Indeed, it follows from
\cite[Exercise~6.7]{rw} that
\begin{equation*}
T\big(\bar{x};M(t)\big)=\big\{w\in\mathbb{R}^n\big|\;\big\langle\nabla_x f(t,\bar{x}),w\big\rangle\in T\big(f(t,\bar{x});\mathbb{R}_{-}\big)\big\},\quad t\in{\cal T}_f,
\end{equation*}
which implies in turn the representation
\begin{align*}
T\big(\bar{x},M(t)\big)=\big\{w\in\mathbb{R}^n\big|\;\big\langle\nabla_x f(t,\bar{x}),w\big\rangle\le 0\big\},
\end{align*}
and hence $\bigcap_{t\in{\cal T}_f}T(\bar{x};M(t))=\{w|\;\langle\nabla_x f(t,\bar{x}),w\rangle\le 0\;\mbox{ for all }\;t\in{\cal T}_f\}$. Taking now any dense set $A\subset{\cal T}_f$, we have that $\bigcap_{t\in A}M(t)=\bigcap_{t\in{\cal T}_f}M(t)$. Furthermore, the continuity of $t\mapsto\nabla_x f(t,\bar{x})$ ensures that
\begin{align*}
&T\Big(\bar{x};\bigcap_{t\in A} M(t)\Big)=T\Big(\bar{x};\bigcap_{t\in{\cal T}_f}M(t)\Big)=\bigcap_{t\in{\cal T}_f}T\big(\bar{x};M(t)\big)\\
&=\big\{w\in\mathbb{R}^n\big|\;\big\langle\nabla_x f(t,\bar{x}),w\big\rangle\le 0\;\mbox{ for all }\;t\in A\big\}=\bigcap_{t\in A}T\big(\bar{x};M(t)\big),
\end{align*}
which verifies the CHIP assumption of Theorem~\ref{firstresultCHIP}. Applying finally Theorem~\ref{firstresultCHIP} to \eqref{IOPTIMIZATION.FUZZY.SIP.3} and arguing as in the proof of Corollary~\ref{Theorem01:inequalitysystem}, we finish the verification of both assertions of the theorem.
\end{proof}\vspace*{0.05in}
	
To conclude, let us present a useful consequence of Theorem~\ref{theoremFinal}(i), where the  CHIP assumption is automatically satisfied.\vspace*{-0.05in}
\begin{corollary}[optimality conditions for semi-infinite programs with linear inequality constraints]\label{sip-lim} Let $\bar{x}$ be a local minimizer of the problem:
\begin{align*}
\begin{array}{c}
\mbox{minimize }\;h(x)\;\mbox{ subject to }\;\langle a(t),x\rangle\le b(t)\;\mbox{ for all }\;t\in{\cal T},
\end{array}
\end{align*}
where $a\colon{\cal T}\to\mathbb{R}^n$ and $b\colon{\cal T}\to\mathbb{R}$ are continuous functions with $a(t)\ne 0$ for all $t\in{\cal T}_f$ from Theorem~{\rm\ref{theoremFinal}} with $f(t,x):=\la a(t),x\ra-b(t)$. Assume that $\ox\in{\rm int}\bigcap_{t\in{\cal T}\backslash{\cal T}_f}\{x\in\mathbb{R}^n|\;\langle a(t),x\rangle<b(t)\}$. Then we have
\begin{align*}
0\in\nabla h(\bar{x})+{\rm cl}\Big(\int_{{\cal T}_f}{\rm cone}\big\{a(t)\big\}d\nu(t)\Big)\;\mbox{ for any }\;\nu\in\mathfrak{M}_{+}({\cal T}_f).
\end{align*}
\end{corollary}\vspace*{-0.05in}
\begin{proof} $\;$ It follows directly from Theorem~\ref{theoremFinal}(i), where the normal regularity assumption is replaced by the Fr\'echet differentiability of $h$ by Remark~\ref{diss}(d), and  the CHIP assumption holds by Corollary~\ref{PropositionCHIP}.
\end{proof}\vspace*{0.05in}

We refer the reader to \cite{clmp2,gl,lnp,m18,mn,mn1,mn2} and the bibliographies therein for various qualification conditions that lead to the possibility to avoid the closure operation in necessary optimality conditions of type \eqref{inclusiongradient} for particular forms of semi-infinite programs with inequality constraints.\vspace*{0.1in}

{\bf Acknowledgements}. The authors are very grateful to three anonymous referees for their numerous comments and remarks that helped us to essentially improve the original presentation.  

\end{document}